\documentclass[11pt]{article}
\usepackage{amssymb,amsmath}
\usepackage[thmmarks,amsmath,amsthm]{ntheorem}
\usepackage{amsfonts, psfrag,enumitem,indentfirst,float,geometry,color,enumerate}
\usepackage{graphicx,algorithm,algpseudocode,pifont}
\usepackage[hidelinks]{hyperref}

\usepackage{epstopdf}
\usepackage{stmaryrd}
\usepackage[normalem]{ulem}
\usepackage{xcolor}

\usepackage{booktabs}
\usepackage{siunitx}
\usepackage{subcaption}
\usepackage[title]{appendix}

\allowdisplaybreaks[4]
\usepackage{lineno}
\numberwithin{equation}{section}

\usepackage{lipsum}

\graphicspath{{figs/}}

\newtheorem{theorem}{Theorem}[section]
\newtheorem{lemma}{Lemma}[section]

\newtheorem{remark}{Remark}[section]
\newtheorem{example}{Example}
\newcommand{\commentout}[1]{{}} 

\newcommand{\vertiii}[1]{{\left\vert\kern-0.25ex\left\vert\kern-0.25ex\left\vert #1
    \right\vert\kern-0.25ex\right\vert\kern-0.25ex\right\vert}}

\newcommand{\jump}[1]{[\![#1]\!]}
    
\newcommand{\bfu}{\mathbf{u}}
\newcommand{\bfv}{\mathbf{v}}
\newcommand{\bff}{\mathbf{f}}
\newcommand{\bfS}{\mathbf{S}}
\newcommand{\bfx}{\mathbf{x}}
\newcommand{\bfn}{\mathbf{n}}

\newcommand{\bfI}{\mathbf{I}}

\usepackage{tikz}
\usepackage{tikz-3dplot}

\allowdisplaybreaks 
\begin{document}
	\title{A Structural Relationship Between Crouzeix–Raviart Immersed Finite Elements for Elliptic and Stokes Interface Problems}
    
	\author{Jiaying Ren$^{1}$, Guozhu Yu$^{1}$, Xu Zhang$^{2,*}$\\
		{\itshape\small $^1$ School of Mathematics, Southwest Jiaotong University, Chengdu, Sichuan 610031, China}\\
		{\itshape\small $^2$ Department of Mathematics, Oklahoma State University, Stillwater OK, 74078, USA}\\
	} 
	\renewcommand{\thefootnote}{\fnsymbol{footnote}}
	\footnotetext[1]{Corresponding author. \\
    {\itshape\small Emails:} 
    rjy586@my.swjtu.edu.cn (J. Ren), 
    yuguozhumail@163.com (G. Yu), 
    xzhang@okstate.edu (X. Zhang).\\
    G. Yu is partially supported by National Natural Science Foundation of China (Grant No. 12561071) and X. Zhang is partially supported by the National Science Foundation (Grant No. DMS-2110833).}
	\date{}
	\maketitle
	\begin{abstract}
		In this paper, we study the structural relationship between immersed finite element (IFE) spaces for scalar elliptic and Stokes interface problems on unfitted meshes. Although IFE methods have been widely developed for individual partial differential equation models, the algebraic connection between scalar and vector IFE spaces has not been systematically explored. We establish a precise unisolvence relationship between the immersed Crouzeix–Raviart (CR) element for elliptic interface problems and the immersed CR-$P_0$ element for Stokes interface problems in both gradient and stress formulations. By analyzing the block structure of the local IFE matrices, we show that the determinant of the Stokes IFE matrix factorizes in terms of the determinant of the corresponding elliptic IFE matrix. This factorization reveals an intrinsic algebraic link between the two classes of immersed elements and explains the unisolvence of the Stokes IFE spaces through the scalar elliptic case. The same result extends naturally to three dimensions and significantly simplifies the corresponding analysis. Numerical experiments confirm optimal convergence rates for both velocity and pressure in three-dimensional Stokes interface problems. 

	\end{abstract}
    
	\section{Introduction}
    Let $\Omega \subset \mathbb{R}^d$ ($d=2,3$) be an open bounded domain divided by a smooth interface $\Gamma$ into two disjoint subdomains $\Omega^+$ and $\Omega^-$ such that $\overline{\Omega} = \overline{\Omega^+ \cup \Gamma\cup\Omega^-}$. Let $\mu(\bfx)$ be a piecewise constant function defined by
	\begin{equation}
		\mu(\bfx) = 
		\begin{cases} 
			\mu^+ & \text{in } \Omega^+, \\
			\mu^- & \text{in } \Omega^-,
		\end{cases}
	\end{equation}
    with $\mu^\pm>0$. 
We consider the following three interface problems. \\

\noindent\textbf{(P0): Scalar elliptic interface problem}
\begin{subequations}
   \begin{align}
	- \nabla \cdot(\mu(\bfx)\nabla u)
	&=f ~~\text{~in~} \Omega^+ \cup \Omega^-, \label{eq:Elliptic} \\
	\jump{u}_\Gamma &= 0 ~~\text{~on~} \Gamma,
	\label{eq:Elliptic-Continuity-1} \\
	\jump{\mu\nabla u\cdot \bfn}_\Gamma &= 0 ~~\text{~on~} \Gamma, 
	\label{eq:Elliptic-Continuity-2} \\
	u &= 0 ~~\text{~on~} \partial\Omega. \label{eq:Elliptic-BoundaryCondition}
\end{align} 
\end{subequations}

\noindent\textbf{(P1): The Stokes interface problem in gradient (Laplacian) form} 
\begin{subequations}\begin{align}
-\mu \Delta\bfu + \nabla p 
&=\bff ~~\text{~in~} \Omega^+ \cup \Omega^-, \label{eq:Stokes-grad-1} \\
\nabla \cdot \bfu &= 0 ~~\text{~in~} \Omega, \label{eq:Stokes-grad-2} \\
\jump{\bfu}_\Gamma &= \mathbf{0} ~~\text{~on~} \Gamma,
\label{eq:Stokes-Continuity-grad-1} \\
\jump{(\mu\nabla\bfu-p\mathbb{I})\bfn}_\Gamma &= \mathbf{0} ~~\text{~on~} \Gamma, \label{eq:Stokes-Continuity-grad-2} \\
\bfu &= \mathbf{0} ~~\text{~on~} \partial\Omega. \label{eq:BoundaryCondition-grad}
\end{align}
\end{subequations}

\noindent\textbf{(P2): The Stokes interface problem in stress (symmetric gradient) form} 
\begin{subequations}\begin{align}
-\nabla \cdot \sigma(\bfu,p) 
&=\bff ~~\text{~in~} \Omega^+ \cup \Omega^-, \label{eq:Stokes-epsi-1} \\
\nabla \cdot \bfu &= 0 ~~\text{~in~} \Omega, \label{eq:Stokes-epsi-2} \\
\jump{\bfu}_\Gamma &= \mathbf{0} ~~\text{~on~} \Gamma,
\label{eq:Stokes-Continuity-epsi-1} \\
\jump{\sigma(\bfu, p)\bfn}_\Gamma &= \mathbf{0} ~~\text{~on~} \Gamma, \label{eq:Stokes-Continuity-epsi-2} \\
\bfu &= \mathbf{0} ~~\text{~on~} \partial\Omega, \label{eq:BoundaryCondition-epsi}
\end{align}
\end{subequations}
where $\sigma$ is the stress tensor defined by
\begin{equation}\label{eq:Relationship-grad-epsi}
    \sigma(\bfu,p) = 2\mu \varepsilon(\bfu) - p\mathbb{I},~~~
    \varepsilon(\bfu) = \frac{1}{2}(\nabla \bfu + (\nabla \bfu)^T).
\end{equation}

In the Stokes interface problems (P1) and (P2), $\bfu$ and $p$ denote the velocity and pressure, respectively. Here $\mathbb{I}$ is the identity matrix, and $\bfn$ denotes the unit normal vector on $\Gamma$ pointing from $\Omega^+$ to $\Omega^-$. For a vector-valued function $\bfv$, the jump across the interface is defined by 
\[\jump{\bfv}_\Gamma := \bfv^+|_\Gamma - \bfv^-|_\Gamma\] 
where $\bfv^\pm := \bfv|_{\Omega^\pm}$. The same notation is used for scalar and matrix-value functions. 

Although (P1) and (P2) are equivalent within each subdomain under the incompressibility constraint, they differ in their interface jump conditions. The stress formulation (P2) is physically more natural, whereas the gradient formulation (P1) is often more convenient for analysis. Clarifying the relation between these two formulations is important for the structural analysis of immersed finite element spaces.


    Elliptic and Stokes interface problems arise in multi-material and multi-phase systems, including incompressible two-phase flows, composite materials, and porous media applications. Across the interface, discontinuous coefficients lead to reduced regularity of solutions, often producing non-smooth behavior such as kinks or singularities. Standard finite element methods typically require interface-fitted meshes in order to achieve optimal convergence \cite{Babuska2000, Xu1982, chen1998finite}. However, generating such meshes can be costly when the interface geometry is complicated or evolving.
    
    Immersed finite element (IFE) methods overcome this difficulty by allowing interface to cut through a background mesh that is not aligned with $\Gamma$. 
    The key idea is to modify the local shape functions on interface elements so that the jump conditions are built into the approximation space, while the mesh itself remains unchanged. Since Li introduced the IFE method for one-dimensional elliptic problems in \cite{Li1998}, the method has been extensively developed for higher-dimensional elliptic interface problems; see for example    \cite{2018AdjeridBenromdhaneLin,2023AdjeridBabuskaGuoLin,He2019a, 2004LiLinLinRogers,2021GuoZhang}. Various conforming and nonconforming IFE spaces have also been constructed and analyzed; see \cite{Guo2019a,Lin2015a,Vallaghe2010, 2019LinSheenZhang, 2020GuoLin}. In particular, Guo and Lin \cite{Guo2019c} developed a unified framework in two dimensions covering conforming linear, nonconforming Crouzeix–Raviart (CR), bilinear ($Q_1$), and nonconforming rotated-$Q_1$ elements,
    \textcolor{black}{
    the unisolvence result is obtained based on Cartesian triangular meshes; Ji et al. \cite{ji2023analysis} generalized the results of nonconforming CR and rotated-$Q_1$ to general interface-unfitted meshes. 
    The three-dimensional nonconforming CR element was analyzed in \cite{ji2023immersed}, the unisolvence of basis functions was also derived on arbitrary tetrahedrons without any angle restrictions.}
    
    For Stokes interface problems, the construction of IFE space is more challenging due to the coupling between velocity and pressure and the incompressibility constraint. Adjerid et al. \cite{Adjerid2015} proposed an immersed $Q_1$-$Q_0$ discontinuous Galerkin method and proved optimal convergence. Later, Jones and Zhang \cite{JonesZhang2021} developed nonconforming CR-$P_0$ IFE spaces for two-dimensional Stokes interface problems and subsequently extended the approach to solve Navier-Stokes equations \cite{2022WangZhangZhuang}. Ji et al. \cite{ji2022immersed} proved the unisolvence of the corresponding CR-$P_0$ IFE basis functions on arbitrary triangles and established optimal approximation capabilities. High-order Taylor-Hood IFE spaces were developed and analyzed in \cite{2021ChenZhang, 2024ChenZhang}. 
      \textcolor{black}{A mini IFE method on Cartesian meshes was also presented in \cite{ji2025mini}, and the corresponding theory including the unisolvence was established for both two- and three-dimensions.}

    Despite the substantial literature on IFE spaces for individual PDE models, especially elliptic and Stokes interface problems, the algebraic relationship between the corresponding scalar and vector IFE spaces has not been systematically studied. In this paper, we investigate the structural relationship between the immersed CR-$P_0$ element for the Stokes systems (P1)-(P2) and the immersed CR element for the scalar elliptic interface problem (P0). This connection clarifies the algebraic foundation of Stokes IFE spaces, reveals determinant factorizations linking scalar and vector construction. Moreover, it simplifies the analysis in higher dimensions, particularly for the three-dimensional CR-$P_0$ immersed element.

    We establish a unified framework connecting scalar elliptic and Stokes IFE spaces in two steps. First, we relate the elliptic problem (P0) to Stokes problem in gradient form (P1) through their common algebraic structure of the gradient terms, which yields a determinant factorization of the local IFE matrix. Second, we bridge the stress formulation (P2) to (P1) by using the identity between the Laplacian and the divergence of the stress tensor under incompressibility. This approach leads to the structural chain
    \begin{center}
        Elliptic (P0) $\longrightarrow$ Stokes (P1) $\longrightarrow$  Stokes (P2)
    \end{center}
    through which the unisolvence of the Stokes immersed spaces can be analyzed in terms of scalar elliptic building blocks. We further develop the CR-$P_0$ immersed element spaces for three-dimensional Stokes interface problems and show that its unisolvence can be reduced to the scalar elliptic case likewise. 
    
    Our main result shows that the determinant of the local Stokes IFE matrix admits a factorization in terms of the scalar elliptic IFE matrix. More precisely
    \begin{equation}
    \det \big(M_1(\mu^+,\mu^-)\big) = \det \big(M_0(\mu^+,\mu^-)\big)^{d-1} \det \big(M_0(1,1)\big), \qquad d = 2,3. 
    \end{equation}
    and
    \begin{equation}
    \det \big(M_2(\mu^+,\mu^-)\big) = \det \big(M_1(\mu^+,\mu^-)\big).
    \end{equation} 
    Consequently, the unisolvence of the Stokes IFE spaces follows directly from the corresponding scalar elliptic result.

	The remainder of the paper is organized as follows. In Section \ref{Sec:CR-P0}, we recall the immersed CR element for elliptic interface problems and the CR-$P_0$ IFE spaces for Stokes interface problems. Section \ref{Sec:Unisolvence} presents the determinant factorization and the corresponding unisolvence analysis in two dimensions. In Section \ref{sec: 3D}, we develop the CR-$P_0$ IFE spaces for three-dimensional Stokes interface problems and establish their unisolvence by linking them to the scalar elliptic CR immersed element. Section \ref{Sec:NumericalResults} reports numerical experiments for three-dimensional Stokes interface problems. Finally, Section \ref{Sec:Conclusions} contains concluding remarks.

\section{IFE Spaces for Elliptic and Stokes Interface Problems}
\label{Sec:CR-P0}
In this section, we recall the immersed finite element spaces for the scalar elliptic interface problem (P0) and the Stokes interface problems (P1)-(P2) in two dimensions.

\subsection{Preliminaries}
Let $\mathcal{T}_h$ be a shape-regular triangulation of $\Omega$, not necessarily aligned with the interface $\Gamma$. An element $T\in\mathcal{T}_h$ is called an interface element if $T\cap \Gamma \ne \emptyset$; otherwise, it is called a non-interface element. The sets of interface elements and non-interface elements are denoted by $\mathcal{T}_h^i$ and $\mathcal{T}_h^n$, respectively.  For each element $T \in \mathcal{T}_h$, let 
$h_T = \text{diam}(T)$ and define $h:=\max_{T\in\mathcal{T}_h}h_T$. 

Let $\mathcal{F}_h$ denote the set of all faces of the mesh, namely, edges in two dimensions and faces in three dimensions. Let $\mathcal{F}_h^o$ and $\mathcal{F}_h^b$ be the sets of interior and boundary faces, respectively. For each face $F \in \mathcal{F}_h$, define its diameter by $h_F$.
The sets of interface and non-interface faces are denoted by $\mathcal{F}_h^i$ and $\mathcal{F}_h^n$, respectively.

As in \cite{2004LiLinLinRogers,Lin2015a}, we assume the following conditions on the mesh-interface configuration:
\begin{description}
    \item[(H1)]  The interface $\Gamma$ intersects at most two edges of any triangle. 
    \item[(H2)]  Each edge intersects the interface at most once, unless the edge lies on the interface. 
\end{description}

More specifically, let $T$ be an interface triangle with vertices $A_i$, $i=1,2,3$. Assume that $\Gamma$ intersects $T$ at two points, denoted by $D$ and $E$. We use the line segment $\Gamma_{h,T} := \overline{DE}$ to approximate the actual interface $\Gamma_T = \Gamma \cap T$. The element $T$ is subdivided by $\Gamma_{h,T}$ into two sub-elements \textcolor{black}{$T_h^+$ and $T_h^-$}. See Figure \ref{fig:InterfaceTriangle} for an illustration. 
\textcolor{black}{Given a function $v$, the jump across the approximate interface $\Gamma_{h,T}$ is defined by 
\[\jump{v}_{\Gamma_{h,T}} := v^+|_{\Gamma_{h,T}}- v^-|_{\Gamma_{h,T}}\] 
with $v^\pm:=v|_{T_h^\pm}$. Let $\bar{\bfn}=(\bar{n}_1, \bar{n}_2)^T$ be the unit normal vector of $\Gamma_{h,T}$ pointing toward $T_h^-$. 
According to the approximate interface $\Gamma_{h,T}$, we let $\mu_h(\bfx)|_{T_h^{\pm}}=\mu^{\pm}$}.

\textcolor{black}{On each interior face $F = \partial T_1 \cap \partial T_2$ with $T_1, T_2 \in \mathcal{T}_h$, let $\bfn_F$ be the unit normal vector of $F$ pointing toward $T_2$. For
a piecewise smooth function $v$, we define the jump and average across the face $F$ by 
\[
[v]_F := v|_{T_1} - v|_{T_2},~~~\{v\}_F := \frac12(v|_{T_1} + v|_{T_2}).
\]
If $F \in \mathcal{F}_h^b$, then $\bfn_F$ denotes the unit normal vector of $F$ outward to $\Omega$, and we 
define $[v]_F := v$ and $\{v\}_F := v$. 
For vector or matrix-valued functions, the notations $[\cdot]_F$, $\{\cdot\}_F$ and $\jump{\cdot}_{\Gamma_{h,T}}$ are defined analogously. 
}

\begin{figure}[htbp]
  \centering
  \includegraphics[scale=1]{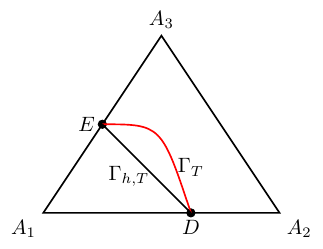}
  \caption{An interface triangular element. The red curve $\Gamma_T$ denotes the actual interface, and the line segment $\Gamma_{h,T}$ is its linear approximation inside the element.}
  \label{fig:InterfaceTriangle}
\end{figure}

\subsection{Scalar CR IFE space for (P0)}
We recall the Crouzeix–Raviart (CR) immersed finite element space for the elliptic interface problem. 

On a non-interface element $T \in \mathcal{T}_h^n$, the local CR space $S_{h0}^n(T)$ coincides with $P_1(T)$, the space of linear polynomials. The local degrees of freedom $N_{i,T}$ 
are given by the edge averages on $T$, i.e., 
\begin{equation}\label{eq:scalar-dof}
	N_{i,T}(v) = \frac{1}{|F_i|} \int_{F_i} v ds, \quad i=1,2,3.
\end{equation} 

On an interface element $T\in \mathcal{T}_h^i$, the local IFE space consists of piecewise linear functions defined separately on \textcolor{black}{$T_h^+$ and $T_h^-$}, which satisfies the interface jump conditions
\begin{subequations}
\begin{align}
	\jump{v}_{\Gamma_{h,T}} = 0, 
	\label{eq:Continuity-scalar-1}\\
	\jump{\textcolor{black}{\mu_h\nabla v\cdot\bar{\bfn}}}_{\Gamma_{h,T}} = 0.   
	\label{eq:Continuity-scalar-2}
\end{align}
\end{subequations}

The local IFE space on an interface element $T$ is defined by
\begin{equation*}
	S_{h0}^i(T) = 
	\{v: v|_{\textcolor{black}{T_h^{\pm}}}=v^{\pm}, ~v^{\pm}\in P_1(\textcolor{black}{T_h^\pm)}, \text{~and~} v \text{~satisfies~} \eqref{eq:Continuity-scalar-1}-\eqref{eq:Continuity-scalar-2}\}.
\end{equation*}
The global nonconforming CR IFE space for the elliptic interface problem (P0) is defined by
\begin{equation}
	 S_{h0}(\mathcal{T}_h) = 
	\left\{ v : v|_T \in S_{h0}^n(T) \ \forall T \in \mathcal{T}_h^n,\ 
	            v|_T \in S_{h0}^i(T) \ \forall T \in \mathcal{T}_h^i,
	~\text{and} \int_F \textcolor{black}{[v]_F} ds = 0\ \forall F \in \mathcal{F}_h^o\right\}.
\end{equation}

\subsection{Vector CR-$P_0$ IFE spaces for (P1)-(P2)}

Next, we construct CR-$P_0$ immersed finite element spaces for the Stokes interface problems (P1) and (P2). The velocity is approximated by nonconforming CR element, while the pressure is approximated by piecewise constants. 

For a non-interface element $T \in \mathcal{T}_h^n$, we use the standard CR-$P_0$ pair
\[\bfS_h^n(T)=P_1(T)^2\times P_0(T).\]
The degrees of freedom for velocity are the edge averages of each component, and the degree of freedom for the pressure is the element average. More precisely, 
\begin{equation*}
N_{i,T}(\bfv,q) = \frac{1}{|F_i|} \int_{F_i} v_1 ds,\ \quad N_{i+3,T}(\bfv,q) = \frac{1}{|F_i|} \int_{F_{i}} v_2 ds, \qquad i =1,2,3,
\end{equation*}
\begin{equation*}
\end{equation*}
and 
\begin{equation*}
N_{7,T}(\bfv,q)=\frac{1}{|T|} \int_T q d\bfx,
\end{equation*}
where $\bfv=(v_1, v_2)^T$. 
Therefore, the total number of local degrees of freedom is $7$.

For coupled systems such as the Stokes problems (P1) and (P2), the velocity and pressure are linked through the interface jump conditions. Accordingly, the local IFE spaces must be defined in a coupled manner. Let $\bfv^{\pm}\in P_1(\textcolor{black}{T_h^\pm)^2}$ and $q^{\pm}\in P_0(\textcolor{black}{T_h^\pm})$. We impose the following discrete interface jump conditions.

For the gradient formulation  (P1), we impose  
\begin{align}
  \jump{\bfv}_{\Gamma_{h,T}} = 0, 
     \label{eq:Continuity-grad-1}\\
  \jump{(\textcolor{black}{\mu_h\nabla \bfv-q\mathbb{I})\bar\bfn}}_{\Gamma_{h,T}} = 0,            \label{eq:Continuity-grad-2}\\
  \jump{\nabla \cdot \bfv}_{\Gamma_{h,T}} = 0.
     \label{eq:Continuity-grad-3}
\end{align}
The corresponding local IFE space is denoted by 
\begin{equation}\label{eq: local space 1}
\mathbf{S}_{h1}^i(T) = 
\{(\bfv, q): \bfv^{\pm}\in P_1(\textcolor{black}{T_h^\pm)^2},  q^{\pm}\in P_0(\textcolor{black}{T_h^\pm}), \text{~and~}  (\bfv, q) \text{~satisfies~} \eqref{eq:Continuity-grad-1}-\eqref{eq:Continuity-grad-3} \}.
\end{equation}

For the Stokes system in stress formulation {(P2)}, we replace the interface condition \eqref{eq:Continuity-grad-2} with
\begin{equation}\label{eq:Continuity-epsi-2}
     \jump{\textcolor{black}{(2\mu_h\varepsilon(\bfv)-q\mathbb{I})\bar\bfn}}_{\Gamma_{h,T}} = 0.
\end{equation}
This defines the local IFE space 
\begin{equation}\label{eq: local space 2} 
\mathbf{S}_{h2}^i(T) = 
\{(\bfv, q): \bfv^{\pm}\in P_1(\textcolor{black}{T_h^\pm)^2},  q^{\pm}\in P_0(\textcolor{black}{T_h ^\pm}), \text{~and~} (\bfv, q) \text{~satisfies~} \eqref{eq:Continuity-grad-1}, \eqref{eq:Continuity-epsi-2},~ \text{and} ~\eqref{eq:Continuity-grad-3} \}.
\end{equation}

The local IFE spaces $\mathbf{S}_{h1}^i(T)$ and $\mathbf{S}_{h2}^i(T)$
are different because they are defined using different traction jump conditions. 

The global CR-$P_0$ IFE spaces for the Stokes problems (P1) and (P2) are defined by
\begin{equation}\label{eq: global space}
\begin{split}
    \bfS_{hj}(\mathcal{T}_h) &= 
\bigg\{ (\bfv, q) : (\bfv, q)|_T \in \bfS_h^n(T) \ \forall T \in \mathcal{T}_h^n,\ 
(\bfv, q)|_T \in \bfS_{hj}^i(T) \ \forall T \in \mathcal{T}_h^i,\\
&\hspace{40mm}  \text{~and~} \int_F \textcolor{black}{[\bfv]_F} ds = 0\ \forall F \in \mathcal{F}_h^o\bigg\}\qquad j=1,2.  
\end{split}
\end{equation}

\section{Unisolvence Relationships in Two Dimensions}
\label{Sec:Unisolvence}

In this section, we establish the algebraic relationship between the scalar CR immersed finite element for the elliptic interface problem and the CR-$P_0$ immersed finite element for the Stokes interface problems in two dimensions. The key tool is a block factorization of the local IFE matrices.

\subsection{Unisolvence of Scalar CR IFE Functions for Elliptic Interface Problem (P0)}
As stated in Section 2.2, the local IFE shape functions on an interface element $T$ have the following piecewise linear form:
\begin{equation}\label{eq:scalar-function-2D-1}
  	v|_{\textcolor{black}{T_h^\pm}}(x,y)=a^{\pm}+b^{\pm}x+c^{\pm}y.
    \end{equation}   
The function satisfies the interface jump conditions \eqref{eq:Continuity-scalar-1} and \eqref{eq:Continuity-scalar-2}. 
The following unisolvence result is well known for two-dimensional elliptic interface problems; see \textcolor{black}{\cite{Guo2019c,ji2023analysis}}. 
  \begin{lemma}\label{lemma:Ellipitc-unisolvency}
  	The nonconforming IFE shape functions of the form \eqref{eq:scalar-function-2D-1} are uniquely determined by the prescribed edge values \eqref{eq:scalar-dof} together with the jump conditions \eqref{eq:Continuity-scalar-1}-\eqref{eq:Continuity-scalar-2}, regardless of the interface location and the coefficient jump.
  \end{lemma}
  
Under the continuity condition \eqref{eq:Continuity-scalar-1}, the IFE function in \eqref{eq:scalar-function-2D-1} can be written in the form
 \begin{equation}
 v(x,y)=\left\{
\begin{array}{ll}
 a+bx+cy+mL(x,y), &\text{~if~}(x,y)\in \textcolor{black}{T_h^+},\\
 a+bx+cy, &\text{~if~}(x,y)\in \textcolor{black}{T_h^-},\\
\end{array}
\right.
 \label{eq:scalar-function-2D-2}
 \end{equation}
where $L(x,y) = \textcolor{black}{\bar n}_1(x-x_D)+\textcolor{black}{\bar n}_2(y-y_D)$ is the signed-distance linear function associated with the interface approximation $\Gamma_{h,T}$, and $(x_D,y_D)$ denotes the coordinate of the intersection point $D$.

The {coefficient vector} 
\[\mathbf{c} = (a,b,c,m)^T\]  
is uniquely determined by the degree of freedom \eqref{eq:scalar-dof} together with the flux jump condition \eqref{eq:Continuity-scalar-2}. More specifically, the resulting linear system can be written as 
\[M_0(\mu^+,\mu^-) \mathbf{c} =\mathbf{w},\]
where $M_0(\mu^+,\mu^-)\in\mathbb{R}^{4\times 4}$ is the local coefficient matrix  and $\mathbf{w}$ collects the prescribed edge averages. 

The matrix $M_0(\mu^+,\mu^-)$ can be decomposed as  
\[M_0(\mu^+,\mu^-) = \left(\begin{array}{c} A\\C(\mu^+,\mu^-)\end{array}\right),\]
where $A\in\mathbb{R}^{3\times 4}$ corresponds to the edge-average degrees of freedom and is
independent of the coefficients $\mu^\pm$, while $C\in\mathbb{R}^{1\times 4}$ arises from the flux jump condition and therefore depends on $\mu^\pm$. More precisely,
\begin{equation}\label{eq: matrix C 2D}
C(\mu^+,\mu^-) =
\begin{pmatrix}
0, & (\mu^+-\mu^-)\textcolor{black}{\bar n}_1, & (\mu^+-\mu^-)\textcolor{black}{\bar n}_2, & \mu^+\\
\end{pmatrix}.    
\end{equation}
Thus, the coefficient matrix $M_0$ depends on $\mu^\pm$ only through the last row $C(\mu^+,\mu^-)$, while the block $A$ is purely geometric. 

The following lemma is a direct consequence of the unisolvence result in \textcolor{black}{\cite{Guo2019c,ji2023analysis}} and will  serve as the key building block for the Stokes analysis. 
\begin{lemma}\label{lemma:Scalar-Unisolvence-2D}
For any admissible interface configuration and any $\mu^\pm>0$, 
\[\det\big(M_0(\mu^+,\mu^-)\big)\neq 0.\]
Hence, the scalar CR immersed finite element is unisolvent.
\end{lemma}

\subsection{Unisolvence of CR-$P_0$ IFE functions for Stokes Interface Problems}
We now analyze the CR-$P_0$ immersed finite element for the two-dimensional Stokes interface problems (P1) and (P2). In the local basis $(v_1,v_2,q)$, the velocity and pressure unknowns are coupled through the traction continuity condition \eqref{eq:Continuity-grad-2} for (P1). 

Due to the continuity of velocity components \eqref{eq:Continuity-grad-1}, each velocity component $v_k$ ($k=1,2$) can be written as 
\begin{equation*}
 v_k(x,y)=\left\{
\begin{array}{ll}
 a_k+b_kx+c_ky+m_kL(x,y), &\text{~if~}(x,y)\in T_h^+,\\
 a_k+b_kx+c_ky, &\text{~if~}(x,y)\in T_h^-,\\
\end{array}
\right.
 \end{equation*}
where $L$ is the signed-distance linear function associated with
$\Gamma_{h,T}$. The pressure is approximated by a piecewise constant function
\[q(x,y)=q^\pm, \text{~in~} \textcolor{black}{T_h^\pm}.\]  
As in the elliptic case, we define the unknown vector
\[
\mathbf c :=
(\mathbf c_1;\mathbf c_2;\mathbf q)
\in\mathbb{R}^{10},\]
where
\[
\mathbf c_k:=(a_k,b_k,c_k,m_k)^T\in\mathbb{R}^4,\qquad k=1,2,\]
and 
\[
\mathbf q:=(q^+,q^-)^T\in\mathbb{R}^2.
\]
Suppose we order the equations as follows:  
(i) three velocity edge-average conditions for $v_1$; 
(ii) the first component of the traction jump condition;
(iii) three velocity edge-average conditions for $v_2$;
(iv) the second component of the traction jump condition;
(v) the divergence continuity condition; and 
(vi) the pressure element-average condition. 
The local system can be written as
\begin{equation}\label{eq: stokes system 1}
    M_1(\mu^+,\mu^-) \mathbf{c} = \mathbf{w},
\end{equation}
where
\begin{equation}\label{eq: M1 2D}
M_1(\mu^+,\mu^-) = 
\begin{pmatrix}
A & 0  & 0 & 0\\
C(\mu^+,\mu^-) & 0  & -\textcolor{black}{\bar n}_1 & \textcolor{black}{\bar n}_1\\
0 & A  & 0 & 0\\
0 & C(\mu^+,\mu^-) & -\textcolor{black}{\bar n}_2 & \textcolor{black}{\bar n}_2 \\
D_1 & D_2  & 0 & 0\\
0 & 0 & \textcolor{black}{\frac{|T^+_h|}{|T|}} & \textcolor{black}{\frac{|T^-_h|}{|T|}}\\
\end{pmatrix}.
\end{equation}
We observe that the divergence jump condition \eqref{eq:Continuity-grad-3} yields 
\begin{equation*}
m_1\textcolor{black}{\bar n}_1+m_2\textcolor{black}{\bar n}_2 = 0, \label{eq:Continuity-grad-3-2Dmatrix}
\end{equation*}
and hence
\begin{equation}\label{eq: D1D2}
    D_1 = (0,0,0,\textcolor{black}{\bar n}_1)=\textcolor{black}{\bar n}_1\mathbf{e}_4^T,~~~\text{and}~~~D_2 = (0,0,0,\textcolor{black}{\bar n}_2) = \textcolor{black}{\bar n}_2\mathbf{e}_4^T. 
\end{equation}

We are now ready to state the determinant factorization for the Stokes IFE matrices. 
\begin{theorem}\label{thm:det-2d}
The determinant of the Stokes IFE matrix $M_1$ satisfies
\begin{equation}\label{eq: det M1}
    \det \big(M_1(\mu^+,\mu^-)\big)
=
\det \big(M_0(\mu^+,\mu^-)\big)\,
\det \big(M_0(1,1)\big).
\end{equation}
Consequently, if the scalar elliptic CR-IFE element is unisolvent, then the two-dimensional Stokes CR-$P_0$ immersed finite element is also unisolvent.
\end{theorem}

\begin{proof}
Adding the second to last column to the last column of $M_1$,
we obtain a matrix whose last column is $\mathbf{e}_{10}$. Let $\widetilde M_1$ be the minor of this matrix associated with the bottom right entry. Then 
\begin{equation}\label{eq: M1 tilde}
   \widetilde M_1:=
   \begin{pmatrix}
A & 0  & 0\\
C(\mu^+,\mu^-) & 0  & -\textcolor{black}{\bar n}_1\\
0 & A   & 0\\
0 & C(\mu^+,\mu^-) & -\textcolor{black}{\bar n}_2 \\
D_1 & D_2  & 0 
\end{pmatrix} = 
\begin{pmatrix}
M_0(\mu^+,\mu^-) & 0 & -D_1^T\\
0 & M_0(\mu^+,\mu^-)  & -D_2^T\\
D_1 & D_2  & 0 
\end{pmatrix}.
\end{equation}

Since $M_0(\mu^+,\mu^-)$ is invertible, we apply the Schur complement to $M = \text{diag}(M_0,M_0)$ to have
\begin{equation}\label{eq: det M1 tilde}
    \begin{split}
        \det(\widetilde M_1) &= \det(M_0)^2\det\left(0-(D_1,D_2)M^{-1}\begin{pmatrix} -D_1^T \\ -D_2^T \end{pmatrix}\right) \\
        &=\det(M_0)^2\det(D_1M_0^{-1}D_1^T + D_2M_0^{-1}D_2^T)\\
        &=\det(M_0)^2(\mathbf{e}_4^TM_0^{-1}\mathbf{e}_4).
\end{split}
\end{equation}
The last step is due to \eqref{eq: D1D2} and the fact that $\textcolor{black}{\bar n}_1^2+\textcolor{black}{\bar n}_2^2 = 1$. Note that $\mathbf{e}_4^TM_0^{-1}\mathbf{e}_4$ is the bottom right entry of $M_0^{-1}$. By Cramer's rule, we have 
\[
\mathbf{e}_4^TM_0^{-1}\mathbf{e}_4 = \frac{(-1)^{4+4}cof_{44}(M_0)}{\det(M_0)},
\]
where $cof_{ij}(M_0)$ denotes the $(j,i)$-th cofactor of $M_0$. Since $cof_{44}(M_0)$ depends on the first three rows and first three columns of $M_0$, hence; it is independent of $C(\mu^+,\mu^-)$. On the other hand, when $\mu^+ = \mu^- = 1$, the last row of $M_0(1,1)$ becomes 
\[
C(1,1) = (0,0,0,1),
\]
so that $\det(M_0(1,1)) = cof_{44}(M_0)$. Thus, 
\[
\mathbf{e}_4^TM_0^{-1}\mathbf{e}_4 = \frac{\det(M_0(1,1))}{\det(M_0(\mu^+,\mu^-))}.
\]
Substituting this into \eqref{eq: det M1 tilde}, we obtain \eqref{eq: det M1}. This completes the proof. In particular, the local Stokes IFE system is algebraically reducible to the scalar elliptic case.
\end{proof}

Now we consider the unisolvence of the IFE space $\bfS_{h2}$ for Stokes interface problem (P2) in stress formulation. The only condition that changes is the traction jump condition $\jump{\textcolor{black}{(\mu_h\nabla\bfv-q\mathbb{I})\bar \bfn}}_{{\Gamma}_{h,T}}$ is replaced with 
$\jump{\textcolor{black}{(2\mu_h\varepsilon(\bfv)-q\mathbb{I})\bar \bfn}}_{{\Gamma}_{h,T}}$. 
\textcolor{black}{
Due to the relation 
\begin{equation}
\jump{(\textcolor{black}{2\mu_h\varepsilon(\bfv)-q\mathbb{I})\bar\bfn}}_{{\Gamma}_{h,T}} 
= \jump{(\textcolor{black}{\mu_h\nabla\bfv-q\mathbb{I})\bar\bfn}}_{{\Gamma}_{h,T}} + \jump{\textcolor{black}{\mu_h(\nabla\bfv)^T\bar\bfn}}_{{\Gamma}_{h,T}},
\end{equation}
and the form of $\bfv=(v_1, v_2)$ that
\begin{equation*}
 v_k(x,y)=\left\{
\begin{array}{ll}
 v_k^+(x,y)=a_k+b_kx+c_ky+m_kL(x,y), &\text{~if~}(x,y)\in T_h^+,\\
  v_k^-(x,y)=a_k+b_kx+c_ky, &\text{~if~}(x,y)\in T_h^-.\\
\end{array}
\right.
 \end{equation*}
 By straightforward calculation, we have
\begin{equation*}
\jump{\mu_h(\nabla\bfv)^T\bar \bfn}_{{\Gamma}_{h,T}}
=(\mu^+-\mu^-)\left((\nabla\bfv^-)^T\bar \bfn\right)+
\mu^+\left(\nabla\begin{pmatrix}
m_1L(x,y)\\  
m_2L(x,y)
\end{pmatrix}\right)^T\textcolor{black}{\bar \bfn},
\end{equation*}
and
\begin{equation*}
\left(\nabla\begin{pmatrix}
m_1L(x,y)\\  
m_2L(x,y)
\end{pmatrix}\right)^T\bar{\bfn}=
\begin{pmatrix}
(m_1\bar{n}_1+m_2\bar{n}_2)\bar{n}_1\\  
(m_1\bar{n}_1+m_2\bar{n}_2)\bar{n}_2
\end{pmatrix}=
\jump{\nabla\cdot\bfv}_{\Gamma_{h,T}}\bar{\bfn}  =0,
\end{equation*}
which lead to
\begin{equation}
\jump{\mu_h(\nabla\bfv)^T\bar \bfn}_{{\Gamma}_{h,T}}
=(\mu^+-\mu^-)\left((\nabla\bfv^-)^T\bar \bfn\right).
\end{equation}
 }

The local IFE coefficient matrix for Problem (P2) can be written as
\textcolor{black}{
\begin{equation}
M_2(\mu^+,\mu^-) = 
\begin{pmatrix}
A & 0 & 0 & 0\\
C &0 & -\textcolor{black}{\bar n}_1 & \textcolor{black}{\bar n}_1\\
0 & A & 0 & 0\\
0 & C & -\textcolor{black}{\bar n}_2 & \textcolor{black}{\bar n}_2 \\
D_1 & D_2 & 0 & 0\\
0 & 0 & \textcolor{black}{\frac{|T^+_h|}{|T|}} & \textcolor{black}{\frac{|T^-_h|}{|T|}}
\end{pmatrix}
+(\mu^+-\mu^-)
\begin{pmatrix}
0 & 0 & 0 & 0\\
C_{11} &C_{12} &0 &0\\
0 & 0 & 0 & 0\\
C_{21} &C_{22} &0 & 0 \\
0 &0 & 0 & 0\\
0 & 0 &0 &0
\end{pmatrix},  
\end{equation}
}
where
\[
C_{11} = \begin{pmatrix} 0, &\textcolor{black}{\bar n}_1, &0, &0 \end{pmatrix} = \textcolor{black}{\bar n}_1\mathbf{e}_2^T\qquad
C_{12} = \begin{pmatrix} 0, &\textcolor{black}{\bar n}_2, &0, &0 \end{pmatrix} = \textcolor{black}{\bar n}_2\mathbf{e}_2^T.
\]
\[
C_{21} = \begin{pmatrix} 0, & 0, &\textcolor{black}{\bar n}_1, &0 \end{pmatrix} = \textcolor{black}{\bar n}_1\mathbf{e}_3^T\qquad
C_{22} = \begin{pmatrix} 0, & 0, &\textcolor{black}{\bar n}_2, &0 \end{pmatrix} = \textcolor{black}{\bar n}_2\mathbf{e}_3^T.
\]

The following theorem shows that the unisolvence for two Stokes formulations is equivalent. 

\begin{theorem}[Unisolvence of the Stress Formulation]\label{thm:stress}
Given any positive $\mu^\pm$, there holds
\begin{equation}\label{eq: stokes det equal}
    \det \big(M_2(\mu^+,\mu^-)\big)=\det \big(M_1(\mu^+,\mu^-)\big).
\end{equation}
Consequently, the CR-$P_0$ immersed finite element in the stress formulation is also unisolvent.
\end{theorem}

\begin{proof}
Recall that the matrices $M_1(\mu^+,\mu^-)$ and $M_2(\mu^+,\mu^-)$ differ only by the
additional row contributions
\[
\textcolor{black}{(\mu^+-\mu^-)}(C_{11},C_{12},0,0), \qquad \textcolor{black}{(\mu^+-\mu^-)}(C_{21},C_{22},0,0),
\]
where
\begin{equation}\label{eq: C vectors}
   (C_{11},C_{12})=(\textcolor{black}{\bar n}_1\mathbf{e}_2^T,\textcolor{black}{\bar n}_2\mathbf{e}_2^T), \qquad
(C_{21},C_{22})=(\textcolor{black}{\bar n}_1\mathbf{e}_3^T,\textcolor{black}{\bar n}_2\mathbf{e}_3^T). 
\end{equation}

Note that
$D_1=\textcolor{black}{\bar n}_1\mathbf{e}_4^T$, $D_2=\textcolor{black}{\bar n}_2\mathbf{e}_4^T.$
Since the scalar IFE matrix
$
M_0(1,1)=\begin{pmatrix}A\\ \mathbf{e}_4^T\end{pmatrix}
$
is invertible, the vectors $\mathbf{e}_2^T$ and $\mathbf{e}_3^T$ lie in the row space of $M_0(1,1)$. Hence there exist vectors $\mathbf{h}_1,\mathbf{h}_2\in\mathbb{R}^{3}$
and scalars $k_1,k_2\in\mathbb{R}$ such that
\[
\mathbf{e}_2^T=\mathbf{h}_1^T A + k_1 \mathbf{e}_4^T, \qquad
\mathbf{e}_3^T=\mathbf{h}_2^T A + k_2 \mathbf{e}_4^T.
\]
Multiplying these vectors by $\textcolor{black}{\bar n}_1$ and $\textcolor{black}{\bar n}_2$, then vectors in \eqref{eq: C vectors} can be written as 

\[
(C_{11},C_{12})=(\textcolor{black}{\bar n}_1\mathbf{h}_1^T,\textcolor{black}{\bar n}_2\mathbf{h}_1^T) 
  \begin{pmatrix}  A &0\\ 0 &A \end{pmatrix}+k_1(D_1,D_2), \qquad
(C_{21},C_{22})=(\textcolor{black}{\bar n}_1\mathbf{h}_2^T,\textcolor{black}{\bar n}_2\mathbf{h}_2^T)  
 \begin{pmatrix}  A &0\\ 0 &A \end{pmatrix} +k_2(D_1,D_2).
\]

It follows that the new row contributions $\textcolor{black}{(\mu^+-\mu^-)}(C_{11},C_{12},0,0)$ and $\textcolor{black}{(\mu^+-\mu^-)}(C_{21},C_{22},0,0)$ in $M_2$
are linear combinations of the rows of $M_1$
\[(A,0,0,0), ~~(0,A,0,0), ~~\text{and} ~~(D_1,D_2,0,0).\]
Therefore, by elementary row operations, $M_2(\mu^+,\mu^-)$ can be reduced to
$M_1(\mu^+,\mu^-)$ without changing the determinant. Hence
\[
\det(M_2(\mu^+,\mu^-))=\det(M_1(\mu^+,\mu^-)).
\]

The conclusion follows from Theorem~3.1.
\end{proof}

\begin{remark}
The Theorem \ref{thm:stress} shows that the extra term $\jump{\textcolor{black}{\mu_h(\nabla\bfv)^T\bar\bfn}}_{{\Gamma}_{h,T}}$ in the stress formulation does not introduce a new independent constraint. Instead, it can be represented as a linear combination of the existing CR degrees of freedom and the divergence constraint.  
\end{remark}

\begin{remark}
Existing works on CR-$P_0$ immersed finite elements for Stokes interface problems, such as \cite{JonesZhang2021,ji2022immersed}, establish unisolvence through direct analysis of the local linear systems. In contrast, the present work reveals a previously unrecognized algebraic structure linking the Stokes and scalar elliptic IFE spaces. In particular, the determinant factorization shows that the unisolvence of the Stokes element can be reduced to that of the scalar elliptic element, thereby providing a unified and structurally transparent explanation.
\end{remark}

\begin{remark}
It was observed in \cite{JonesZhang2021} that the rotated $Q_1$-$Q_0$ IFE spaces yield identical determinants for both the gradient and stress formulations of Stokes interface problems. The analysis developed in this work extends to the rotated $Q_1$-$Q_0$ IFE spaces defined on rectangular meshes for Stokes interface problems. In particular, Theorem~\ref{thm:det-2d} and Theorem~\ref{thm:stress} remain valid in this setting.

\end{remark}

\section{Three-Dimensional IFE Spaces for Stokes Problems}
\label{sec: 3D}
In this section, we develop the CR-$P_0$ immersed finite element spaces for the three-dimensional Stokes interface problems in both gradient and stress formulations. We then show that their unisolvence follows from the same structural factorization principle established in two dimensions.

\subsection{CR-$P_0$ IFE Spaces}
The construction follows the same principle as in two dimensions, but additional care is required because of the extra velocity component and the more complicated geometry of interface tetrahedra. 

\textcolor{black}{
   As in \cite{2025GuoZhang}, we impose the following assumptions on the interface-tetrahedron configuration:
\begin{description}
    \item[(H3)]  The interface $\Gamma$ intersects each edge of a tetrahedron at most once.
    \item[(H4)]  On each face of a tetrahedron, the interface $\Gamma$ intersects at most two edges. 
\end{description}
}

Under these assumptions, we need to consider the following two geometric configurations of interface tetrahedra.
\begin{itemize}
    \item Type I: one vertex is separated from the other three. 
    \item Type II: two vertices are separated from the other two. 
\end{itemize}
See Figure \ref{fig:InterfaceTetrahedron} for an illustration. In the Type II case, the four intersection points need not be coplanar. Following \cite{2020GuoLin, 2021GuoZhang, 2025GuoZhang}, we use the maximum-angle criterion to select three of these points so that each interior angle is bounded above by $\theta<\pi$, and we use them to define the approximation plane. The existence of such a selection is guaranteed by Lemma 2.1 in \cite{2025GuoZhang}. These three selected points are still denoted by $D$, $E$, and $F$. For both types, we use the plane $\Gamma_{h,T}=\triangle DEF$ to approximate the actual interface segment $\Gamma\cap T$ inside the element. The tetrahedron $T$ is thereby divided into two subregions, denoted by $\textcolor{black}{T_h^+}$ and $\textcolor{black}{T_h^-}$.
\textcolor{black}{The above planar approximation may lead to a discontinuous discrete interface across a face shared by two neighboring interface elements. An alternative is to use a level-set-based approximation, as in \cite{ji2023immersed}. In that approach, the zero level set of a linear interpolation of the level-set function gives a planar interface approximation inside each tetrahedron and is compatible across shared faces. The determinant-factorization and unisolvence results developed here depend only on the algebraic structure of the local interface constraints, and therefore can also be applied to such level-set-based approximations.
}

\begin{figure}[htbp]
  \centering
  \includegraphics[scale=1]{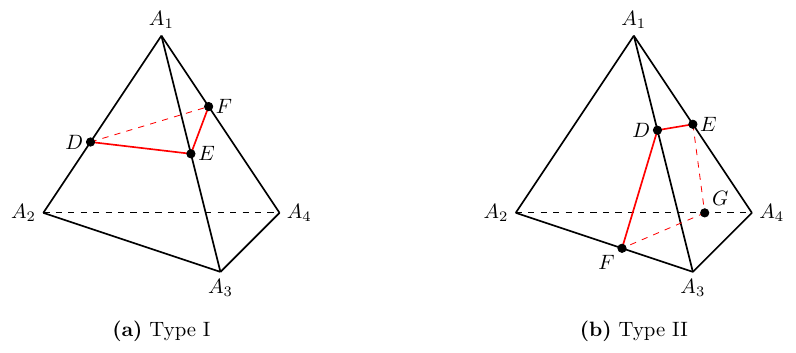}
  \caption{Two typical interface tetrahedral elements.}
  \label{fig:InterfaceTetrahedron}
\end{figure}


On each subregion \textcolor{black}{$T_h^\pm$}, the velocity is approximated by vector-valued linear polynomials, while the pressure is approximated by piecewise constants. Specifically, we write

\[
\mathbf v^\pm(x,y,z)=
\begin{pmatrix}
a_1^\pm+b_1^\pm x+c_1^\pm y+d_1^\pm z \vspace{1mm}\\
a_2^\pm+b_2^\pm x+c_2^\pm y+d_2^\pm z \vspace{1mm}\\
a_3^\pm+b_3^\pm x+c_3^\pm y+d_3^\pm z
\end{pmatrix},
\qquad
q(x,y,z) = q^\pm.
\]

The coefficients are determined by enforcing the CR degrees of freedom
on the faces of the tetrahedron together with the interface jump conditions \eqref{eq:Continuity-grad-1}-\eqref{eq:Continuity-grad-3} or \eqref{eq:Continuity-grad-1}, \eqref{eq:Continuity-epsi-2}, \eqref{eq:Continuity-grad-3} for gradient and stress continuity, respectively. 
The degrees of freedom of the CR-$P_0$ IFE element consist of the face averages of each velocity component,
\[
N_{4(k-1)+i,T}(\mathbf{v},q) = \frac{1}{|F_i|}\int_{F_i} v_k\,ds, \qquad i=1,2,3,4,\quad k=1,2,3,
\]
together with the element average of the pressure,
\[
N_{13,T}(\mathbf{v},q)=\frac{1}{|T|}\int_T q\,d\bfx.
\]
In total, there are $13$ local degrees of freedom in one element.  These conditions produce a local linear system that uniquely determines the coefficients of the immersed finite element shape functions. Figure \ref{fig:ShapeFunction} compares the standard finite element shape function with the CR-$P_0$ IFE shape functions on Type I and Type II interface tetrahedra.

The local IFE spaces $\mathbf{S}_{hj}^i(T)$,  $j=1,2$ are defined in a manner similar to \eqref{eq: local space 1} and \eqref{eq: local space 2} for the two-dimensional Stokes problem in gradient and stress formulation, respectively. The global IFE spaces $\bfS_{hj}(\mathcal{T}_h)$ for three-dimensional case are also defined similarly as in \eqref{eq: global space}.

\begin{figure}[htb!]
    \centering

    \begin{subfigure}{\textwidth}
        \centering
        \includegraphics[width=0.26\textwidth]{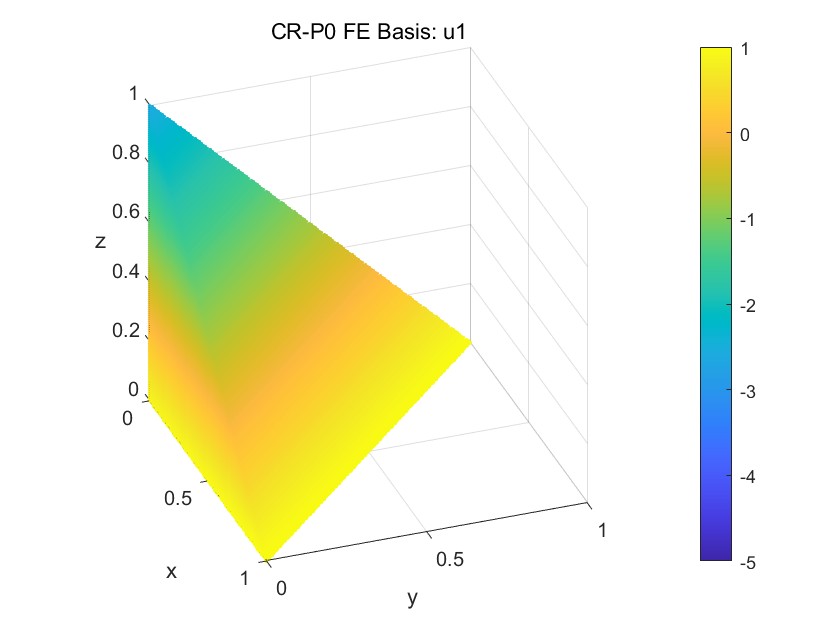}%
        \includegraphics[width=0.26\textwidth]{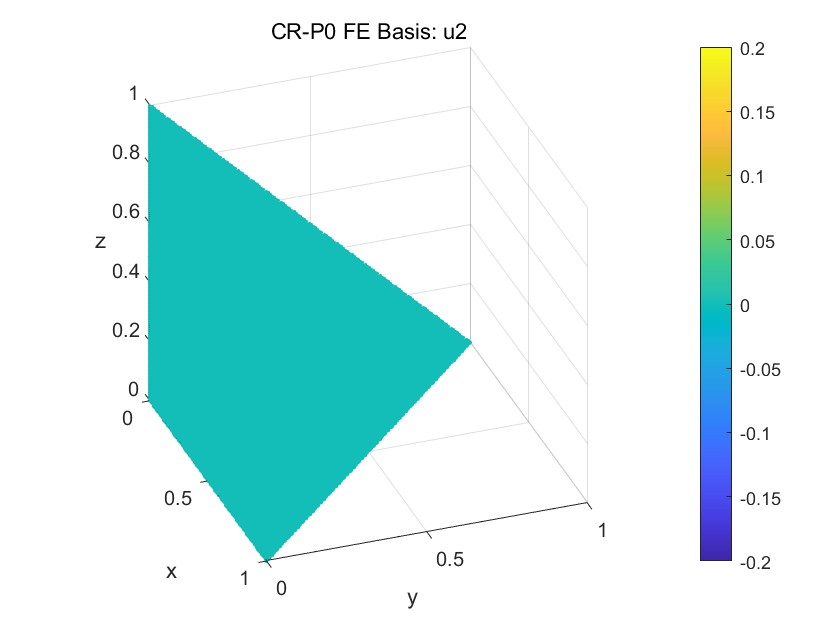}%
        \includegraphics[width=0.26\textwidth]{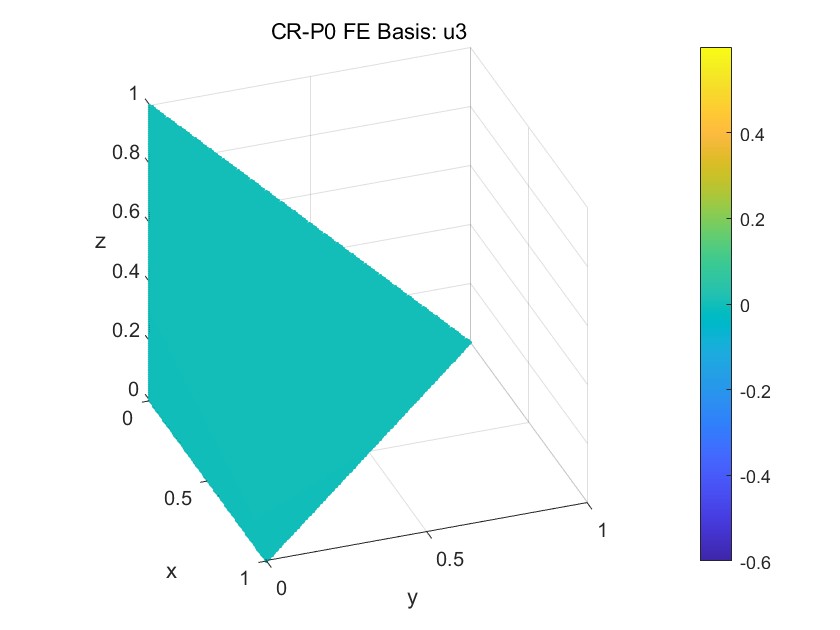}%
        \includegraphics[width=0.26\textwidth]{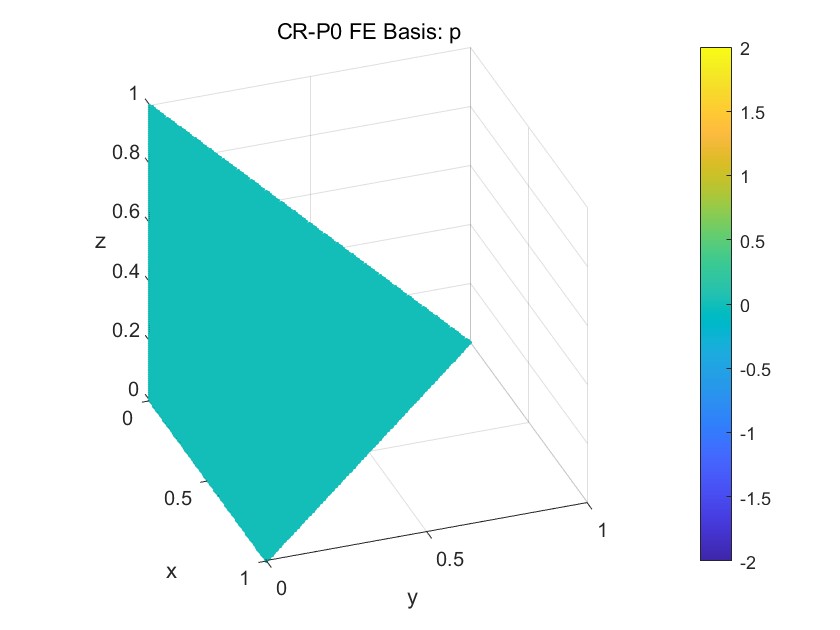}%
        \caption{CR-$P_0$ FE shape function. From left: $u_1,~u_2,~u_3,~p$.}
    \end{subfigure}

    \begin{subfigure}{\textwidth}
        \centering
        \includegraphics[width=0.26\textwidth]{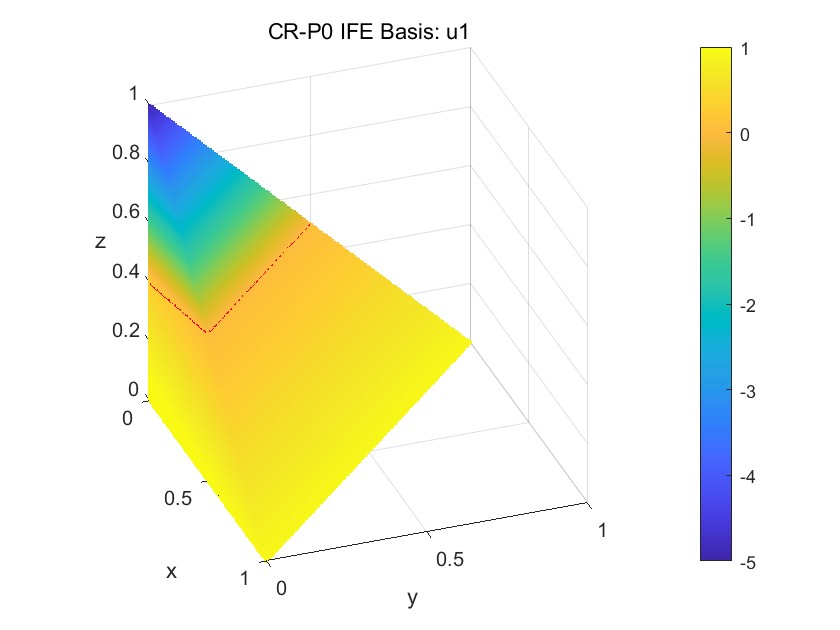}%
        \includegraphics[width=0.26\textwidth]{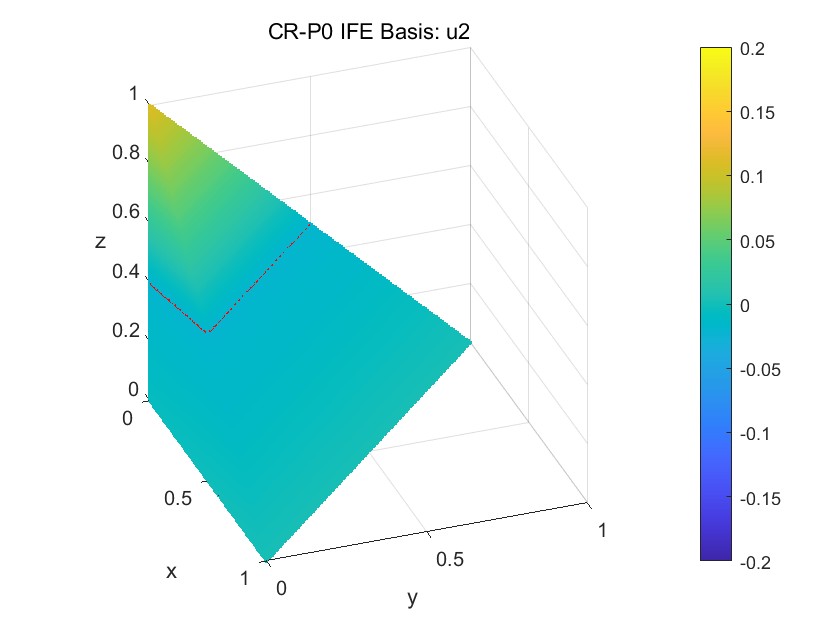}%
        \includegraphics[width=0.26\textwidth]{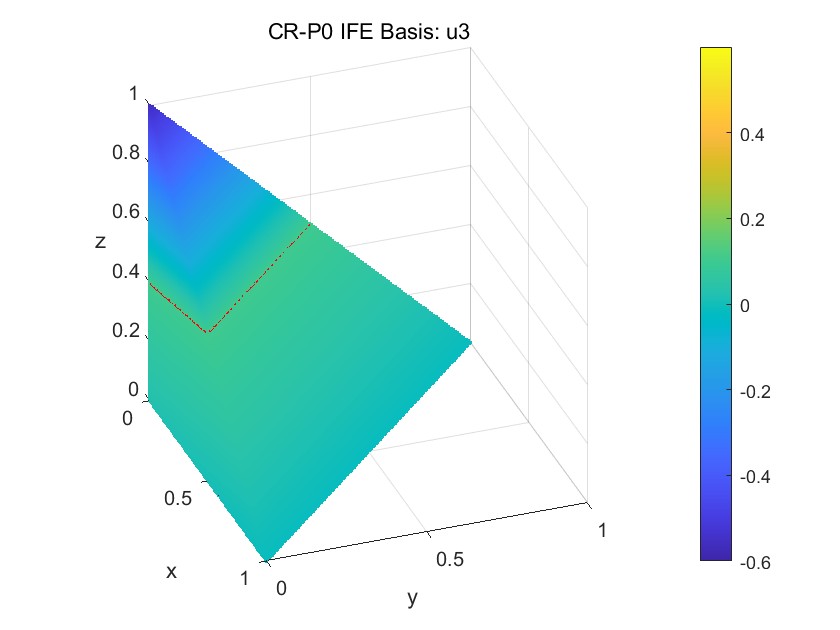}%
        \includegraphics[width=0.26\textwidth]{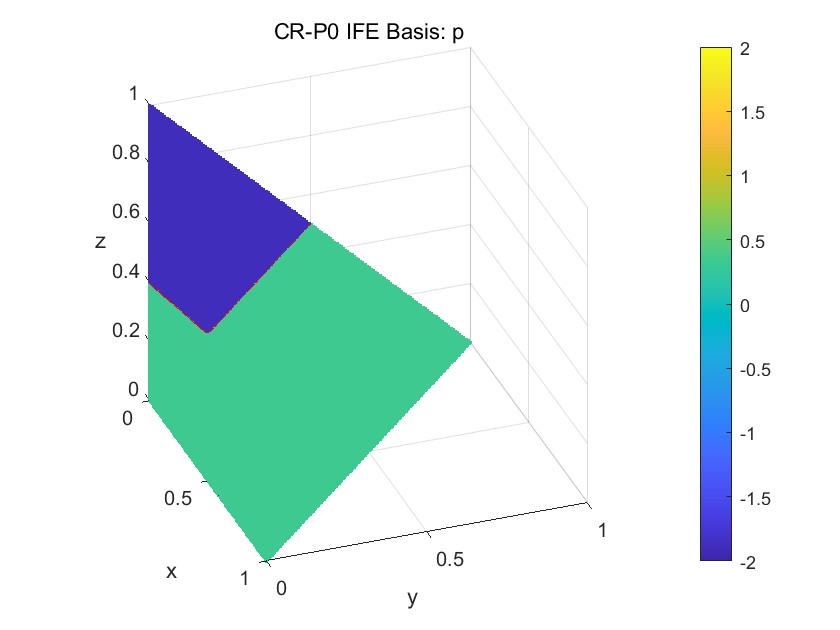}%
        \caption{CR-$P_0$ IFE shape function of Type I. From left: $u_1,~u_2,~u_3,~p$.}
    \end{subfigure}
    \begin{subfigure}{\textwidth}
        \centering
        \includegraphics[width=0.26\textwidth]{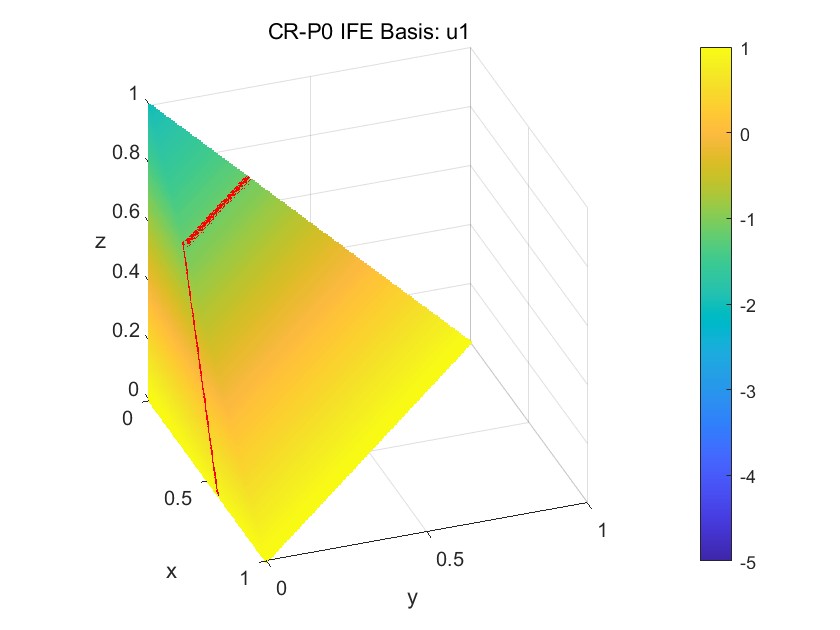}%
        \includegraphics[width=0.26\textwidth]{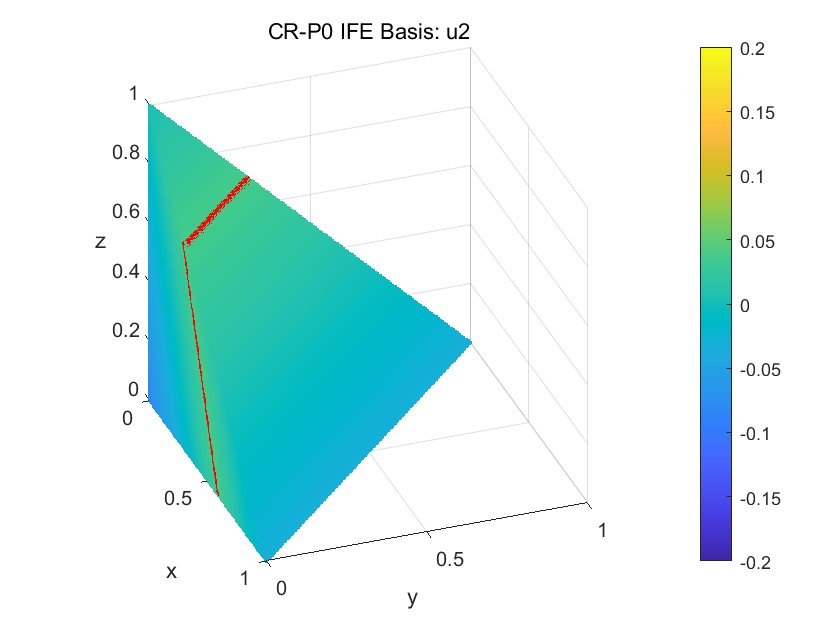}%
        \includegraphics[width=0.26\textwidth]{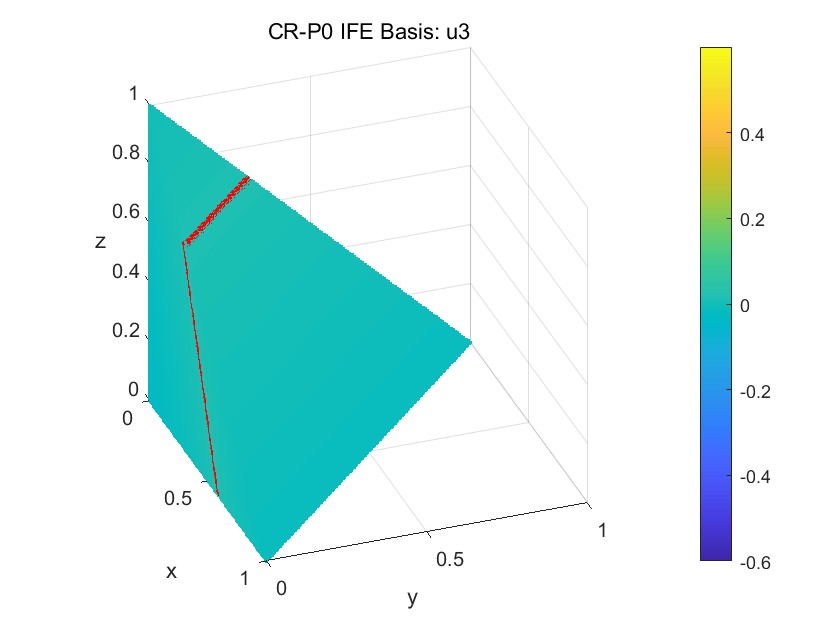}%
        \includegraphics[width=0.26\textwidth]{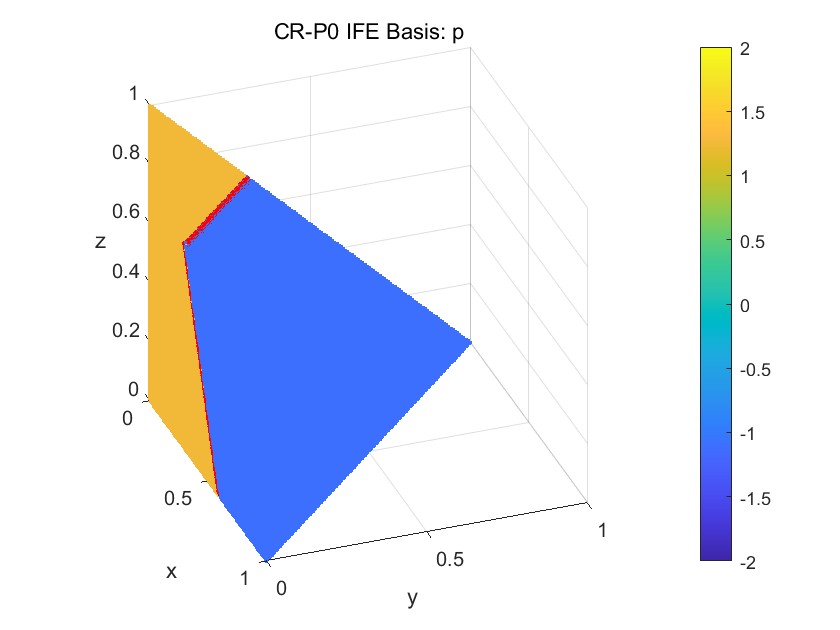}%
        \caption{CR-$P_0$ IFE shape function of Type II. From left: $u_1,~u_2,~u_3,~p$.}
    \end{subfigure}

    \caption{Comparison of the FE shape function and IFE shape functions for problem (P1) corresponding to the degree of freedom $\frac{1}{|F_1|} \int_{F_1} v_1 ds$ with $F_1$ the bottom face. The part that contains the vertex $(0,0,1)$ lies in $\Omega^{-}$, and the remaining part lies in $\Omega^ {+}$. The coefficients are $\mu^- = 1$, $\mu^+ = 5$.}
    \label{fig:ShapeFunction}
\end{figure}

\subsection{Unisolvence}
The local system matrix for the three-dimensional CR-$P_0$ immersed element has a block structure analogous to that in the two-dimensional case. As a result, its determinant can again be reduced to determinants associated with the scalar elliptic IFE system. We now make this relation precise and derive the unisolvence of the three-dimensional Stokes CR-$P_0$ element from the scalar elliptic case.


Under the continuity condition \eqref{eq:Continuity-scalar-1}, the scalar IFE function for elliptic equation in three dimensions can be written as
 \begin{equation}
 v(x,y,z)=\left\{
\begin{array}{ll}
v^+(x,y,z)=a+bx+cy+dz+mL(x,y,z), &\text{~if~}(x,y,z)\in \textcolor{black}{T_h^+},\\
v^-(x,y,z)=a+bx+cy+dz, &\text{~if~}(x,y,z)\in \textcolor{black}{T_h^-},\\
\end{array}
\right.
 \label{eq:scalar-function-3D-2}
 \end{equation}
where $L(x,y,z) = \textcolor{black}{\bar n}_1(x-x_D)+\textcolor{black}{\bar n}_2(y-y_D)+\textcolor{black}{\bar n}_3(z-z_D)$. 

The equation \eqref{eq:scalar-function-3D-2} involves five unknown coefficients, namely $a$, $b$, $c$, $d$, and $m$. These coefficients are determined by the four face-average degrees of freedom together with the flux jump condition \eqref{eq:Continuity-scalar-2}. Accordingly, the local coefficient matrix takes the form
\begin{equation}\label{eq: 3D M0}
    M_0(\mu^+,\mu^-)=\begin{pmatrix} A \\ C \end{pmatrix},
\end{equation}
where $A\in\mathbb{R}^{4\times 5}$ arises from the face-average conditions and $C\in\mathbb{R}^{1\times 5}$ arises from the flux jump condition. More precisely, 
\begin{equation}\label{eq:MatrixC-3D}
C(\mu^+,\mu^-) =
\begin{pmatrix}
0, & (\mu^+-\mu^-)\textcolor{black}{\bar n}_1, & (\mu^+-\mu^-)\textcolor{black}{\bar n}_2,  & (\mu^+-\mu^-)\textcolor{black}{\bar n}_3, & \mu^+\\
\end{pmatrix}.
\end{equation}

The following lemma, which follows from the three-dimensional elliptic IFE analysis in \cite{ji2023immersed}, will serve as the key building block.

\begin{lemma}\label{lemma:Scalar-Unisolvence-3D}
Let $M_0(\mu^+,\mu^-)$ be the coefficient matrix defined in \eqref{eq: 3D M0}. 
Then
\[
\det \big(M_0(\mu^+,\mu^-)\big)\neq 0
\]
for any admissible interface configuration and any $\mu^\pm>0$.
\end{lemma}

Under the continuity condition \eqref{eq:Continuity-grad-1}, each velocity component has the representation
 \begin{equation}
 v_k(x,y,z)=\left\{
\begin{array}{ll}
v_k^+ =a_k+b_kx+c_ky+d_kz+m_kL(x,y,z), &\text{~if~}(x,y,z)\in \textcolor{black}{T_h^+},\\
v_k^- =a_k+b_kx+c_ky+d_kz, &\text{~if~}(x,y,z)\in \textcolor{black}{T_h^-},\\
\end{array}
\right. \quad k=1,2,3,
 \end{equation}
and the pressure is given by
\begin{equation}
 q(x,y,z)=\left\{
\begin{array}{ll}
q^+, &\text{~if~}(x,y,z)\in \textcolor{black}{T_h^+},\\
q^-, &\text{~if~}(x,y,z)\in \textcolor{black}{T_h^-}.\\
\end{array}
\right.
 \end{equation}

For the gradient formulation, the 17 unknown coefficients are determined by 13 local degrees of freedom together with the four interface constraints coming from \eqref{eq:Continuity-grad-2} and \eqref{eq:Continuity-grad-3}. Thus the local system is square.



The corresponding coefficient matrix can be written as
\begin{equation}\label{eq: M1 3D}
M_1(\mu^+,\mu^-) = 
\begin{pmatrix}
A & 0 & 0 & 0 & 0\\
C & 0 & 0 & -\textcolor{black}{\bar n}_1 & \textcolor{black}{\bar n}_1\\
0 & A & 0 & 0 & 0\\
0 & C & 0 & -\textcolor{black}{\bar n}_2 & \textcolor{black}{\bar n}_2 \\
0 & 0 & A & 0 & 0\\
0 & 0 & C & -\textcolor{black}{\bar n}_3 & \textcolor{black}{\bar n}_3 \\
D_1 & D_2 & D_3 & 0 & 0\\
0 & 0 & 0 & \textcolor{black}{\frac{|T^+_h|}{|T|}} & \textcolor{black}{\frac{|T^-_h|}{|T|}}
\end{pmatrix}.
\end{equation}

This matrix is the three-dimensional analogue of the block-structured matrix in Section \ref{Sec:Unisolvence}, with three copies of the scalar elliptic block $M_0 = \begin{pmatrix} A \\ C \end{pmatrix}$ corresponding to the three velocity components.
We observe that the divergence jump condition \eqref{eq:Continuity-grad-3} yields 
\begin{equation}
m_1\textcolor{black}{\bar n}_1+m_2\textcolor{black}{\bar n}_2+m_3\textcolor{black}{\bar n}_3=0, \label{eq:Continuity-grad-3-3Dmatrix}
\end{equation}
and hence
\[
D_1=(0,0,0,0,\textcolor{black}{\bar n}_1),\qquad D_2=(0,0,0,0,\textcolor{black}{\bar n}_2),\qquad D_3=(0,0,0,0,\textcolor{black}{\bar n}_3).
\]

The following result shows that the structural factorization established in two dimensions extends directly to three dimensions.

\begin{theorem}
    [Dimension-independent factorization]
    For the three-dimensional Stokes interface problem in gradient form, the local IFE matrix satisfies
\begin{equation}\label{eq: factor 3}
\det \big(M_1(\mu^+,\mu^-)\big) = \det \big(M_0(\mu^+,\mu^-)\big)^{2} \det \big(M_0(1,1)\big).
\end{equation}
In particular, the corresponding CR-$P_0$ IFE space is unisolvent in three dimensions.
\end{theorem}

\begin{proof}
The matrix in \eqref{eq: M1 3D} contains three copies of the scalar elliptic block $A$ and one pressure block. As in the two-dimensional case, the coefficients $\mu^\pm$ enter only through the flux-jump rows represented by $C$. After eliminating the pressure variables, the remaining coupling is given by the divergence constraint \eqref{eq:Continuity-grad-3-3Dmatrix}, which depends only on the interface geometry and not on $\mu^\pm$. Therefore the Schur complement is identical to that obtained in the case $\mu^+=\mu^-=1$, which yields the factorization \eqref{eq: factor 3}. The unisolvence then follows from Lemma 4.1.
\end{proof}

We next consider the stress formulation and the corresponding space $\bfS_{h2}$. Replacing the gradient-form traction condition \eqref{eq:Continuity-grad-2} by the stress-form condition \eqref{eq:Continuity-epsi-2}, and using
\begin{equation}
	\jump{\textcolor{black}{(2\mu_h\varepsilon(\bfv)-q\mathbb{I})\bar \bfn}}_{\Gamma_{h,T}} 
	= \jump{\textcolor{black}{(\mu_h\nabla\bfv-q\mathbb{I})\bar \bfn}}_{\Gamma_{h,T}} + \jump{\textcolor{black}{\mu_h(\nabla\bfv)^T\bar \bfn}}_{\Gamma_{h,T}},
\end{equation}
we obtain the following local coefficient matrix
\textcolor{black}{
\begin{equation}\label{eq: M2 3}
M_2(\mu^+,\mu^-) = 
\begin{pmatrix}
A & 0 & 0 & 0 & 0\\
C & 0 & 0 & -\bar{n}_1 & \bar{n}_1\\
0 & A & 0 & 0 & 0\\
0 & C & 0 & -{\bar n}_2 & {\bar n}_2 \\
0 & 0 & A & 0 & 0\\
0 & 0 & C & -{\bar n}_3 & {\bar n}_3 \\
D_1 & D_2 & D_3 & 0 & 0\\
0 & 0 & 0 & \frac{|T^+_h|}{|T|} & \frac{|T^-_h  |}{|T|}
\end{pmatrix}
+(\mu^+-\mu^-)\begin{pmatrix}
0 & 0 & 0 & 0 & 0\\
C_{11} & C_{12} & C_{13} & 0 & 0\\
0 & 0 & 0 & 0 & 0\\
C_{21} & C_{22} & C_{23} & 0 & 0 \\
0 & 0 & 0 & 0 & 0\\
C_{31} & C_{32} & C_{33} & 0 & 0 \\
0 & 0 & 0 & 0 & 0\\
0 & 0 & 0 & 0 & 0
\end{pmatrix},
\end{equation}
}
where
\[
C_{11}=(0,\textcolor{black}{\bar n}_1,0,0,0),\qquad 
C_{12}=(0,\textcolor{black}{\bar n}_2,0,0,0),\qquad 
C_{13}=(0,\textcolor{black}{\bar n}_3,0,0,0),
\]
\[
C_{21}=(0,0,\textcolor{black}{\bar n}_1,0,0),\qquad 
C_{22}=(0,0,\textcolor{black}{\bar n}_2,0,0),\qquad 
C_{23}=(0,0,\textcolor{black}{\bar n}_3,0,0),
\]
\[
C_{31}=(0,0,0,\textcolor{black}{\bar n}_1,0),\qquad 
C_{32}=(0,0,0,\textcolor{black}{\bar n}_2,0),\qquad 
C_{33}=(0,0,0,\textcolor{black}{\bar n}_3,0).
\]

This result shows that the factorization is independent of spatial dimension and depends only on the algebraic structure of the interface constraints.
Now we show the equivalence of the stress and gradient formulations in three-dimensional Stokes problems.
\begin{theorem}[Equivalence of gradient and stress formulations]
 For any admissible interface configuration and any $\mu^\pm>0$,
\begin{equation}
   \det(M_2(\mu^+,\mu^-))= \det(M_1(\mu^+,\mu^-)).
 \end{equation}
\end{theorem}

\begin{proof}
The proof is identical in structure to that of Theorem \ref{thm:stress}. The additional terms in \eqref{eq: M2 3} arise from $\jump{\textcolor{black}{\mu_h(\nabla\bfv)^T\bar \bfn}}_{\Gamma_{h,T}}$, under the symmetry-like structure of block $C_{ij}$, these terms can be eliminated by the first, third, fifth and seventh row blocks in $M_2$, and so do not change the determinant. Hence the local matrix for the stress formulation has the same determinant as that for the gradient formulation.
\end{proof}

The above unisolvence results provide the theoretical foundation for the existence of the corresponding IFE spaces. 

\begin{remark}
The three-dimensional result is not obtained by a direct extension of the two-dimensional analysis. Instead, it follows from a dimension-independent factorization principle revealed in Theorem 3.1. This shows that the algebraic structure of the Stokes IFE system is fundamentally governed by scalar elliptic components, independent of spatial dimension.
\end{remark}

\section{Immersed Finite Element Methods and Numerical Examples}
\label{Sec:NumericalResults}

In this section, we examine the accuracy and convergence of the CR-$P_0$ IFE methods for three-dimensional Stokes interface problems through a series of numerical experiments. Since numerical results for two-dimensional Stokes interface problems have already been extensively reported in the literature, see, e.g., \cite{JonesZhang2021,Jo2024,ji2022immersed}, we focus here on the three-dimensional case. 
\textcolor{black}{The non-symmetric formulation as in \cite{ji2022immersed} is adopted in the following numerical tests.}

 We study both the interpolation error and the IFE solution error for several interface geometries and coefficient contrasts. The velocity errors are measured in the $L^2$ norm and the broken $H^1$ seminorm, while the pressure errors are measured in the $L^2$ norm.

 The CR-$P_0$ IFE interpolation operator is defined by
 \[
\bfI_h :H^1(\Omega)^3\times L^2(\Omega)\rightarrow \bfS_h(\mathcal{T}_h)
\]
where $\bfS_h(\mathcal{T}_h)$ denotes either $S_{h1}(\mathcal T_h)$ or $S_{h2}(\mathcal T_h)$ depending on the formulation. For $(\bfu,p)\in H^1(\Omega)^3 \times L^2(\Omega)$, we write
\[
\bfI_h(\bfu,p)=\bfI_h(u_1,u_2,u_3,p) =(\tilde{u}_1, \tilde{u}_2, \tilde{u}_3, \tilde{p}) \in \bfS_h(\mathcal{T}_h),
\]
with
\[
\int_{F_i}\tilde u_1\,ds=\int_{F_i}u_1\,ds,\qquad 
\int_{F_i}\tilde u_2\,ds=\int_{F_i}u_2\,ds,\qquad
\int_{F_i}\tilde u_3\,ds=\int_{F_i}u_3\,ds,\qquad i=1,2,3,4,
\]
and
\[
\int_T \tilde p\,d\bfx=\int_T p\,d\bfx.
\]
The interpolation errors for the velocity components and the pressure are denoted by
\begin{equation*}
e_{k,I}=u_k-\tilde u_k,\qquad k=1,2,3,\qquad e_{p,I}=p-\tilde p.
\end{equation*}

Let $\Omega=[-1,1]^3$. We partition $\Omega$ into $N\times N\times N$ equal cubes and then subdivide each cube into six congruent tetrahedra. All three-dimensional Stokes interface problems are solved on these unfitted Cartesian meshes using both the gradient formulation in Subsection \ref{sec:WeakForm-P1} and the stress formulation in Subsection \ref{sec:WeakForm-P2}.

The experimental convergence rate is defined by
\[
\frac{\log(e_j/e_{j+1})}{\log(h_j/h_{j+1})}
=
\frac{\log(e_j/e_{j+1})}{\log(N_{j+1}/N_j)},
\]
where $e_j$ denotes  $\|e_{\bfu}\|_{L^2(\Omega)}$, $|e_{\bfu}|_{H^1(\Omega)}$, or $\|e_p\|_{L^2(\Omega)}$ on the $N_j\times N_j\times N_j$ mesh.

\subsection{Weak formulation for Problem (P1)}
\label{sec:WeakForm-P1}
For the gradient formulation, the nonconforming CR-$P_0$ IFE method seeks $(\bfu_h,p_h)\in \bfS_{h1}(\mathcal T_h)$ such that
\textcolor{black}{    
\begin{equation}\label{eq:WeakForm-grad}
a_h(\bfu_h,\bfv_h)+b_h(\bfv_h,p_h)-b_h(\bfu_h,q_h)+c_h(p_h,q_h)= (\bff, \bfv_h), \quad \forall (\bfv_h, q_h) \in \bfS_{h1}^0,
\end{equation}
where the bilinear form is defined as follows,
\begin{equation}\label{eq:binliear-grad}
\begin{aligned}
a_h(\bfu_h,\bfv_h)
&:=\sum_{T\in\mathcal{T}_h}\int_T\mu_h\nabla\bfu_h:\nabla\bfv_h d\bfx\\
&\hspace{-1cm}-\sum_{F\in\mathcal{F}_h^i}\int_F\left(\{\mu_h\nabla \bfu_h\bfn_F\}_F\cdot [\bfv_h]_F-\{\mu_h\nabla \bfv_h\bfn_F\}_F\cdot [\bfu_h]_F\right)ds+\sum_{F\in\mathcal{F}_h^i}\frac{1}{h_F}\int_F[\bfu_h]_F\cdot[\bfv_h]_Fds,\\
b_h(\bfu_h,q_h)
&:=-\sum_{T\in\mathcal{T}_h}\int_T q_h\nabla\cdot\bfv_h d\bfx+\sum_{F\in\mathcal{F}_h^i}\int_F\{q_h\}_F [\bfv_h\cdot\bfn_F]_Fds,\\
c_h(p_h,q_h)
&:=\sum_{F\in\mathcal{F}_h^i}h_F\int_F  [p_h]_F[q_h]_Fds,
\end{aligned}
\end{equation}
and the test space is defined by } 
\begin{equation*}
\bfS_{h1}^0(\mathcal{T}_h)
:=
\left\{ (\bfv, q) \in \bfS_{h1}(\mathcal{T}_h) : \int_F \bfv ds = 0\ \forall F \in \mathcal{F}_h^b\right\}.
\end{equation*}

\begin{example}[Planar interface]
   We first consider a three-dimensional Stokes interface problem with a planar interface 
   \[\Gamma = \{(x,y,z): z + {\pi}/{7} = 0 \}.\] 
   The interface divides $\Omega$ into the subdomains
\[
\Omega^-=\{(x,y,z): z<-\pi/7\},\qquad
\Omega^+=\{(x,y,z): z>-\pi/7\}.
\]
The exact velocity and pressure are chosen as
	\begin{align*}
	\bfu &= 
	\begin{cases} 
	\left(\dfrac{x^2}{\mu^{+}} \left( z+\dfrac{\pi}{7}\right),\  
    -\dfrac{y^2}{\mu^{+}} \left( z+\dfrac{\pi}{7}\right),\ 
    \dfrac{y-x}{\mu^{+}} \left( z+\dfrac{\pi}{7}\right)^2\right)^T \quad 
    \text{in} \ \Omega^+,\\
	\left(\dfrac{x^2}{\mu^{-}} \left( z+\dfrac{\pi}{7}\right),\ 
    -\dfrac{y^2}{\mu^{-}} \left( z+\dfrac{\pi}{7}\right),\ 
    \dfrac{y-x}{\mu^{-}} \left( z+\dfrac{\pi}{7}\right)^2\right)^T \quad 
    \text{in} \ \Omega^-, \\
	\end{cases} \\
	p &= 2e^x - e^y - e^z.
	\end{align*} 
\end{example}

Table \ref{tab:planar-grad-interpolation1} reports the interpolation errors and convergence rates. We observe second-order convergence for the velocity in the $L^2$ norm, first-order convergence in the broken $H^1$ semi-norm, and first-order convergence for the pressure in the $L^2$ norm. Table \ref{tab:planar-grad-solution1} reports the corresponding errors for the IFE solution, where the observed convergence rates are again optimal. These results are consistent with those reported for the two-dimensional case.

	\begin{table}[H]
	   \centering
	   \caption{Errors of CR-$P_0$ IFE interpolations for Example 1 with $\mu^- = 10$ and $\mu^+ = 1$.}
	   \label{tab:planar-grad-interpolation1}
	   \begin{tabular}{cccccccc}
		   \toprule
	    	$N$ & \#Dof 
			&$\|e_{\mathbf u,I}\|_{L^2(\Omega)}$ & Rate
			&$\|e_{p,I}\|_{L^2(\Omega)}$ & Rate 
            &$|e_{\mathbf u,I}|_{H^1(\Omega)}$ & Rate\\
         \midrule
           4 & 2976 & 4.4598e-2 & n/a  & 9.2596e-1 & n/a  & 1.1428e+0 & n/a \\
           8 & 22656 & 1.0998e-2 & 2.02 & 4.3743e-1 & 1.08 & 5.7510e-1 & 0.99 \\
		  16 & 176640 & 2.6853e-3 & 2.03 & 2.0163e-1 & 1.12 & 2.8780e-1 & 1.00\\
          24 & 590976 & 1.1789e-3 & 2.03 & 1.2748e-1 & 1.13 & 1.9193e-1 & 1.00\\
		  32 & 1394688 & 6.6018e-4 & 2.01 & 9.3710e-2 & 1.04  & 1.4404e-1 & 1.00\\
        \bottomrule 
    \end{tabular}
	\end{table}
    
    \begin{table}[H]
		\centering
    \textcolor{black}{
		\caption{Errors of CR-$P_0$ IFE solutions for Example 1 with $\mu^- = 10$ and $\mu^+ = 1$.}
		\label{tab:planar-grad-solution1}
		\begin{tabular}{cccccccc}
		\toprule 
		   $N$ & \#Dof 
		   &$\|e_{\mathbf u h}\|_{L^2(\Omega)}$ & Rate
		   &$\|e_{ph}\|_{L^2(\Omega)}$ & Rate 
          &$|e_{\mathbf u h}|_{H^1(\Omega)}$ & Rate\\
		\midrule
		   4 & 2976 & 1.4519e-1 & n/a  & 1.8997e-0 & n/a  & 1.3967e-0 & n/a \\
		   8 & 22656 & 4.0763e-2 & 1.83 & 9.1644e-1 & 1.05 & 7.2282e-1 & 0.95 \\
		  16 & 176640 & 1.0595e-2 & 1.94 & 3.6329e-1 & 1.33 & 3.6627e-1 & 0.98 \\
          24 & 590976  & 4.7204e-3 & 1.99 & 1.9180e-1 & 1.58 & 2.4533e-1 & 0.99 \\
          32 & 1394688 & 2.6654e-3 & 1.99 & 1.2454e-1 & 1.50 & 1.8458e-1 & 0.99 \\
        \bottomrule 
    \end{tabular}
    }
	\end{table}

   \begin{example}[Spherical interface]
	We next consider a problem with a spherical interface 
    \[\Gamma = \{(x,y,z):x^2 + y^2 + z^2 = {\pi}^2/{16}\}.\] 
    The spherical interface separates the domain $\Omega$ into two subdomains \[\Omega^- = \{(x,y,z):x^2 + y^2 + z^2 < {\pi}^2/{16}\},\qquad \Omega^+ = \{(x,y,z):x^2 + y^2 + z^2 > {\pi}^2/{16}\}.\] 
    The exact solutions $\bfu$ and $p$ are chosen as follows:
	\begin{align*}
		\bfu &= 
	\begin{cases} 
	\left(\dfrac{yz}{\mu^{+}} \left(x^2 + y^2 + z^2 - \dfrac{{\pi}^2}{16} \right),\ 
        -\dfrac{xz}{2\mu^{+}} \left(x^2 + y^2 + z^2 - \dfrac{{\pi}^2}{16} \right),\ 
        -\dfrac{xy}{2\mu^{+}} \left(x^2 + y^2 + z^2 - \dfrac{{\pi}^2}{16} \right) \right)^T \quad \text{in} \ \Omega^+, \\
	\left(\dfrac{yz}{\mu^{-}} \left(x^2 + y^2 + z^2 - \dfrac{{\pi}^2}{16} \right),\ 
        -\dfrac{xz}{2\mu^{-}} \left(x^2 + y^2 + z^2 - \dfrac{{\pi}^2}{16} \right),\ 
        -\dfrac{xy}{2\mu^{-}} \left(x^2 + y^2 + z^2 - \dfrac{{\pi}^2}{16} \right) \right)^T \quad \text{in} \ \Omega^-,
	   \end{cases} \label{eq:Example-Spherical-1} \\
		p &= x^3 - y^3 - z^3. 
	\end{align*}
    \end{example}
    The interpolation errors and IFE solution errors for this problem are reported in Tables \ref{tab:spherical-grad-interpolation1} and \ref{tab:spherical-grad-solution1}, respectively. To further examine robustness, Tables \ref{tab:spherical-grad-solution2} and \ref{tab:spherical-grad-solution3} present the IFE solution errors for reversed coefficient contrast and for a large coefficient jump. In all cases, the observed convergence rates remain optimal.
    
    \begin{table}[H]
		\centering
		\caption{Errors of CR-$P_0$ IFE interpolations for Example 2 with $\mu^- = 10$ and $\mu^+ = 1$.}
		\label{tab:spherical-grad-interpolation1}
		\begin{tabular}{cccccccc}
			\toprule
			$N$ & \#Dof 
			&$\|e_{\mathbf u,I}\|_{L^2(\Omega)}$ & Rate
			&$\|e_{p,I}\|_{L^2(\Omega)}$ & Rate 
            &$|e_{\mathbf u,I}|_{H^1(\Omega)}$ & Rate\\
			\midrule
			4  & 2976 & 1.0474e-1 & n/a  & 6.6142e-1 & n/a & 2.3749e+0 & n/a  \\
			8  & 22656 & 2.3615e-2 & 2.15  & 3.1436e-1 & 1.07 & 1.2352e+0 & 0.94  \\
			16 & 176640 & 5.7678e-3 & 2.03  & 1.5349e-1 & 1.03 & 6.2408e-1 & 0.99 \\
			24 & 590976 & 2.5490e-3 & 2.01  & 1.0193e-1 & 1.01 & 4.1692e-1 & 0.99 \\
            32 & 1394688 & 1.4307e-3 & 2.01  & 7.6225e-2 & 1.01 & 3.1279e-1 & 1.00 \\
        \bottomrule 
    \end{tabular}
	\end{table}

    \begin{table}[H]
		\centering
        \textcolor{black}{
		\caption{Errors of CR-$P_0$ IFE solutions for Example 2 with $\mu^- = 10$ and $\mu^+ = 1$.}
		\label{tab:spherical-grad-solution1}
		\begin{tabular}{cccccccc}
			\toprule 
			$N$ & \#Dof 
			&$\|e_{\mathbf u h}\|_{L^2(\Omega)}$ & Rate
			&$\|e_{ph}\|_{L^2(\Omega)}$ & Rate 
            &$|e_{\mathbf u h}|_{H^1(\Omega)}$ & Rate\\
		   \midrule
		   4 & 2976 & 1.9413e-1 & n/a  & 9.3614e-1 & n/a  & 2.6570e+0 & n/a \\
		   8 & 22656 & 5.7706e-2 & 1.75 & 4.1470e-1 & 1.17 & 1.4227e+0 & 0.90 \\
		   16 & 176640 & 1.5161e-2 & 1.93 & 1.9487e-1 & 1.09 & 7.2336e-1 & 0.98 \\
          24 & 590976  & 6.8364e-3 & 1.96 & 1.2497e-1 & 1.10 & 4.8432e-1 & 0.99 \\
          32 & 1394688 & 3.8709e-3 & 1.98 & 9.2188e-2 & 1.06 & 3.6374e-1 & 1.00 \\
        \bottomrule 
    \end{tabular}
    }
	\end{table}

\begin{table}[H]
	\centering
    \textcolor{black}{
	\caption{Errors of CR-$P_0$ IFE solutions for Example 2 with $\mu^- = 1$ and $\mu^+ = 10$.}
	\label{tab:spherical-grad-solution2}
	\begin{tabular}{cccccccc}
		\toprule 
		$N$ & \#Dof 
		&$\|e_{\mathbf u h}\|_{L^2(\Omega)}$ & Rate
		&$\|e_{ph}\|_{L^2(\Omega)}$ & Rate 
		&$|e_{\mathbf u h}|_{H^1(\Omega)}$ & Rate\\
		\midrule
		4 & 2976 & 4.5361e-2 & n/a  & 1.2609e-0 & n/a  & 4.3573e-1 & n/a \\
		8 & 22656 & 1.3247e-2 & 1.78 & 5.0290e-1 & 1.33 & 2.2667e-1 & 0.94 \\
		16 & 176640 & 3.6501e-3 & 1.86 & 2.1995e-1 & 1.19 & 1.2226e-1 & 0.89 \\
		24 & 590976  & 1.6120e-3 & 2.02 & 1.3401e-1 & 1.22 & 8.2905e-2 & 0.96 \\
        32 & 1394688 & 8.9306e-4 & 2.05 & 9.5697e-2 & 1.17 & 6.2436e-2 & 0.99 \\
		\bottomrule 
	\end{tabular}
    }
\end{table}

\begin{table}[H]
	\centering
    \textcolor{black}{
	\caption{Errors of CR-$P_0$ IFE solutions for Example 2 with $\mu^- = 1000$ and $\mu^+ = 1$.}
	\label{tab:spherical-grad-solution3}
	\begin{tabular}{cccccccc}
		\toprule 
		$N$ & \#Dof 
		&$\|e_{\mathbf u h}\|_{L^2(\Omega)}$ & Rate
		&$\|e_{ph}\|_{L^2(\Omega)}$ & Rate 
		&$|e_{\mathbf u h}|_{H^1(\Omega)}$ & Rate\\
		\midrule
		4 & 2976 & 3.8330e-1 & n/a  & 3.9735e+1 & n/a  & 4.5808e+0 & n/a \\
		8 & 22656 & 9.0497e-2 & 2.08 & 3.7486e+0 & 3.41 & 1.6116e+0 & 1.51 \\
		16 & 176640 & 2.3365e-2 & 1.95 & 8.0916e-1 & 2.21 & 7.6307e-1 & 1.08 \\
		24 & 590976  & 9.6373e-3 & 2.18 & 4.7203e-1 & 1.33 & 4.9849e-1 & 1.05 \\
        32 & 1394688 & 5.3689e-3 & 2.03 & 3.2809e-1 & 1.26 & 3.7069e-1 & 1.03 \\
		\bottomrule 
	\end{tabular}
    }
\end{table}

\subsection{Weak formulation for Problem (P2)}
\label{sec:WeakForm-P2}

For the stress formulation, the nonconforming CR-$P_0$ IFE method seeks $(\bfu_h,p_h)\in \bfS_{h2}(\mathcal T_h)$ such that

\textcolor{black}{    
\begin{equation}\label{eq:WeakForm-epsi}
\tilde{a}_h(\bfu_h,\bfv_h)+b_h(\bfv_h,p_h)-b_h(\bfu_h,q_h)+c_h(p_h,q_h)= (\bff, \bfv_h), \quad \forall (\bfv_h, q_h) \in \bfS_{h2}^0,
\end{equation}
with $b_h(\cdot,\cdot)$, $c_h(\cdot,\cdot)$ the same as in \eqref{eq:binliear-grad}, and
\begin{equation*}
\begin{aligned}
\tilde{a}_h(\bfu_h,\bfv_h)
&:=\sum_{T\in\mathcal{T}_h}\int_T 2\mu_h\varepsilon(\bfu_h):\varepsilon(\bfv_h) d\bfx+\sum_{F\in\mathcal{F}_h}\frac{1}{h_F}\int_F[\bfu_h]_F\cdot[\bfv_h]_Fds\\
&\hspace{-1cm}-\sum_{F\in\mathcal{F}_h^i}\int_F\left(\{2\mu_h\varepsilon(\bfu_h)\bfn_F\}_F\cdot [\bfv_h]_F-\{2\mu_h\varepsilon(\bfv_h)\bfn_F\}_F\cdot [\bfu_h]_F\right)ds+\sum_{F\in\mathcal{F}_h^i}\frac{1}{h_F}\int_F[\bfu_h]_F\cdot[\bfv_h]_Fds,\\
\end{aligned}
\end{equation*}
and the test space } 
\begin{equation*}
   \bfS_{h2}^0(\mathcal{T}_h) := \left\{ (\bfv, q) \in \bfS_{h2}(\mathcal{T}_h) : \int_F \bfv ds = 0\ \forall F \in \mathcal{F}_h^b\right\}.
\end{equation*}


\begin{example}[Spherical interface]
	We revisit the spherical interface problem from Example 2, but now solve it using the stress-form weak formulation \eqref{eq:WeakForm-epsi} instead of the gradient-form formulation \eqref{eq:WeakForm-grad}.
  \end{example}

   Since the interpolation errors in this example are very similar to those in Example 2, we report only the IFE solution errors in Table \ref{tab:spherical-epsi-solution1}. The observed convergence rates are again optimal in all reported norms.


	\begin{table}[H]
		\centering
        \textcolor{black}{
		\caption{Errors of CR-$P_0$ IFE solutions for Example 3 with $\mu^- = 10$ and $\mu^+ = 1$.}
		\label{tab:spherical-epsi-solution1}
		\begin{tabular}{cccccccc}
			\toprule
			$N$ & \#Dof 
			&$\|e_{\mathbf u h}\|_{L^2(\Omega)}$ & Rate
			&$\|e_{ph}\|_{L^2(\Omega)}$ & Rate 
            &$|e_{\mathbf u h}|_{H^1(\Omega)}$ & Rate\\
			\midrule
		   4 & 2976 & 1.7743e-1 & n/a & 1.0944e+0 & n/a & 2.5977e+0 & n/a \\
		   8 & 22656 & 4.5852e-2 & 1.95 & 4.5489e-1 & 1.27 & 1.3515e+0 & 0.94 \\
		  16 & 176640 & 1.1460e-2 & 2.00 & 2.0712e-1 & 1.14 & 6.8160e-1 & 0.99 \\
        24 & 590976 & 5.0323e-3 & 2.03 & 1.3182e-1 & 1.11 & 4.5453e-1 & 1.00 \\
        \bottomrule 
    \end{tabular}
    }
	\end{table}

\section{Conclusion}
\label{Sec:Conclusions}

We established a structural relationship between scalar elliptic and Stokes Crouzeix--Raviart immersed finite element spaces. In particular, we showed that the local CR-$P_0$ Stokes IFE matrix admits a determinant factorization in terms of the scalar elliptic CR IFE matrix. This provides a simple and unified explanation of the unisolvence of the Stokes IFE spaces.
The same factorization principle applies to both gradient and stress formulations and extends naturally to three dimensions, reducing the Stokes analysis to the scalar elliptic case. Numerical results confirm optimal convergence for both velocity and pressure.
These results demonstrate that the algebraic structure of Stokes IFE spaces is fundamentally inherited from scalar elliptic components, offering a clear framework for analysis and construction of coupled IFE methods.

\section*{Declarations}

\subsection*{Funding}
\noindent Guozhu Yu is partially supported by National Natural Science Foundation of China (Grant No. 12561071) \\
Xu Zhang is partially supported by the National Science Foundation (Grant No. DMS-2110833)

\subsection*{Conflict of Interest}
\noindent The authors declare that they have no conflict of interest.

\subsection*{Data Availability Statement}
\noindent This study did not generate or analyze any new datasets. All data supporting the findings of this work are contained within the article.


\begin{thebibliography}{10}

\bibitem{Adjerid2015}
S.~Adjerid, N.~Chaabane, and T.~Lin.
\newblock An immersed discontinuous finite element method for {Stokes}
  interface problems.
\newblock {\em Computer Methods in Applied Mechanics and Engineering},
  293:170--190, 2015.

\bibitem{2023AdjeridBabuskaGuoLin}
Slimane Adjerid, Ivo Babu\v~ska, Ruchi Guo, and Tao Lin.
\newblock An enriched immersed finite element method for interface problems
  with nonhomogeneous jump conditions.
\newblock {\em Computer Methods in Applied Mechanics and Engineering},
  404:115770, 2023.

\bibitem{2018AdjeridBenromdhaneLin}
Slimane Adjerid, Mohamed Ben-Romdhane, and Tao Lin.
\newblock Higher degree immersed finite element spaces constructed according to
  the actual interface.
\newblock {\em Computers \& Mathematics with Applications}, 75(6):1868--1881,
  2018.

\bibitem{Babuska2000}
I.~Babuska and J.~E. Osborn.
\newblock Can a finite element method perform arbitrarily badly?
\newblock {\em Mathematics of Computation}, 69(230):443--462, 2000.

\bibitem{2021ChenZhang}
Yuan Chen and Xu~Zhang.
\newblock A {$P_2$}-{$P_1$} partially penalized immersed finite element method
  for {Stokes} interface problems.
\newblock {\em International Journal of Numerical Analysis and Modeling},
  18(1):120--141, 2021.

\bibitem{2024ChenZhang}
Yuan Chen and Xu~Zhang.
\newblock Solving {N}avier-{S}tokes equations with stationary and moving
  interfaces on unfitted meshes.
\newblock {\em Journal of Scientific Computing}, 98(1):19, 2024.

\bibitem{chen1998finite}
Z.~Chen and J.~Zou.
\newblock Finite element methods and their convergence for elliptic and
  parabolic interface problems.
\newblock {\em Numerische Mathematik}, 79(2):175--202, 1998.

\bibitem{Guo2019c}
R.~Guo and T.~Lin.
\newblock A group of immersed finite-element spaces for elliptic interface
  problems.
\newblock {\em IMA Journal of Numerical Analysis}, 39(1):482--511, 2019.

\bibitem{Guo2019a}
R.~Guo and T.~Lin.
\newblock A higher degree immersed finite element method based on a cauchy
  extension for elliptic interface problems.
\newblock {\em SIAM Journal on Numerical Analysis}, 57(4):1545--1573, 2019.

\bibitem{2020GuoLin}
Ruchi Guo and Tao Lin.
\newblock An immersed finite element method for elliptic interface problems in
  three dimensions.
\newblock {\em Journal of Computational Physics}, 414:109478, 2020.

\bibitem{2021GuoZhang}
Ruchi Guo and Xu~Zhang.
\newblock Solving three-dimensional interface problems with immersed finite
  elements: {A}-priori error analysis.
\newblock {\em Journal of Computational Physics}, 441:110445, 2021.

\bibitem{2025GuoZhang}
Ruchi Guo and Xu~Zhang.
\newblock An enriched immersed finite element method for three-dimensional
  interface problems with nonhomogeneous jump.
\newblock {\em submitted, arXiv:2509.12555}, 2025.

\bibitem{He2019a}
C.~He and X.~Zhang.
\newblock Residual-based a posteriori error estimation for immersed finite
  element methods.
\newblock {\em Journal of Scientific Computing}, 81(3):2051--2079, 2019.

\bibitem{ji2023immersed}
H.~Ji.
\newblock An immersed {Crouzeix--Raviart} finite element method in {2D} and
  {3D} based on discrete level set functions.
\newblock {\em Numerische Mathematik}, 153(2):279--325, 2023.

\bibitem{ji2025mini}
H.~Ji, D.~Liang, and Q.~Zhang.
\newblock A mini immersed finite element method for two-phase {Stokes} problems
  on {Cartesian} meshes.
\newblock {\em IMA Journal of Numerical Analysis}, 45(4):1985--2022, 2025.

\bibitem{ji2022immersed}
H.~Ji, F.~Wang, J.~Chen, and Z.~Li.
\newblock An immersed {CR-P0} element for stokes interface problems and the
  optimal convergence analysis.
\newblock {\em Computer Methods in Applied Mechanics and Engineering},
  399:115306, 2022.

\bibitem{ji2023analysis}
Haifeng Ji, Feng Wang, Jinru Chen, and Zhilin Li.
\newblock Analysis of nonconforming {IFE} methods and a new scheme for elliptic
  interface problems.
\newblock {\em ESAIM: Mathematical Modelling and Numerical Analysis},
  57(4):2041--2076, 2023.

\bibitem{Jo2024}
G.~Jo and D~Y. Kwak.
\newblock A new immersed finite element method for two-phase {Stokes} problems
  having discontinuous pressure.
\newblock {\em Computational Methods in Applied Mathematics}, 24(1):49--58,
  2024.

\bibitem{JonesZhang2021}
D.~Jones and X.~Zhang.
\newblock A class of nonconforming immersed finite element methods for {Stokes}
  interface problems.
\newblock {\em Journal of Computational and Applied Mathematics}, 392:113493,
  2021.

\bibitem{Li1998}
Z.~Li.
\newblock The immersed interface method using a finite element formulation.
\newblock {\em Applied Numerical Mathematics}, 27(3):253--267, 1998.

\bibitem{2004LiLinLinRogers}
Z.~Li, T.~Lin, Y.~Lin, and R.~C. Rogers.
\newblock An immersed finite element space and its approximation capability.
\newblock {\em Numerical Methods for Partial Differential Equations},
  20(3):338--367, 2004.

\bibitem{Lin2015a}
T.~Lin, Y.~Lin, and X.~Zhang.
\newblock Partially penalized immersed finite element methods for elliptic
  interface problems.
\newblock {\em SIAM Journal on Numerical Analysis}, 53(2):1121--1144, 2015.

\bibitem{2019LinSheenZhang}
Tao Lin, Dongwoo Sheen, and Xu~Zhang.
\newblock {A nonconforming immersed finite element method for elliptic
  interface problems}.
\newblock {\em Journal of Scientific Computing}, 79(1):442--463, 2019.

\bibitem{Vallaghe2010}
S.~Vallagh\'{e} and T.~Papadopoulo.
\newblock A trilinear immersed finite element method for solving the
  electroencephalography forward problem.
\newblock {\em SIAM Journal on Scientific Computing}, 32(4):2379--2394, 2010.

\bibitem{2022WangZhangZhuang}
Jin Wang, Xu~Zhang, and Qiao Zhuang.
\newblock An immersed {C}rouzeix-{R}aviart finite element method for
  {N}avier-{S}tokes equations with moving interfaces.
\newblock {\em International Journal of Numerical Analysis and Modeling},
  19(4):563--586, 2022.

\bibitem{Xu1982}
J.~Xu.
\newblock Error estimates of the finite element method for the 2nd order
  elliptic equations with discontinuous coefficients.
\newblock {\em Xiangtan University}, 1:1--5, 1982.

\end{thebibliography}


\end{document}